%% file: covario_3d_polytopes_7_arXiv.tex
\documentclass[a4paper]{amsart}
\usepackage{graphicx,color}
\usepackage[psamsfonts]{amssymb}
\usepackage{hyperref}
\usepackage{booktabs}

\newtheorem{theorem}{Theorem}[section]

\newtheorem{lemma}[theorem]{Lemma}
\newtheorem{proposition}[theorem]{Proposition}
\newtheorem{claim}{Claim}[theorem]
\theoremstyle{definition}
\newtheorem{definition}[theorem]{Definition}
\newtheorem{problem}[theorem]{Problem}
\newtheorem{remark}[theorem]{Remark}
\newtheorem{example}[theorem]{Example}

\numberwithin{equation}{section}

\newcommand{\set}[1]{\left\{#1\right\}}

\newcommand{\inte}{{\mathop{\mathrm{int}\,}}}
\newcommand{\cl}{{\mathop{\mathrm{cl}\,}}}
\newcommand{\relbd}{{\mathop{\mathrm{relbd}\,}}}
\newcommand{\relint}{{\mathop{\mathrm{relint}\,}}}
\newcommand{\pos}{{\mathop{\mathrm{pos}\,}}}
\newcommand{\aff}{{\mathop{\mathrm{aff}\,}}}
\newcommand{\conv}{{\mathop{\mathrm{conv}\,}}}
\newcommand{\cone}{{\mathop{\mathrm{cone}\,}}}
\newcommand{\dime}{{\mathop{\mathrm{dim}\,}}}
\newcommand{\supp}{{\mathop{\mathrm{supp}\,}}}
\newcommand{\size}{{\mathop{\mathrm{size}\,}}}
\newcommand{\siz}{{\mathop{\mathrm{sz}\,}}}
\newcommand{\len}{{\mathop{\lambda_1}}}
\newcommand{\area}{{\mathop{\lambda_2}}}
\newcommand{\voln}{{\mathop{\lambda_n}}}
\newcommand{\volt}{{\mathop{\lambda_3}}}

\newcommand{\pa}{{\partial}}

\newcommand{\Real}{\mathbb{R}}
\newcommand{\ee}{{\varepsilon}}
\newcommand{\al}{{\alpha}}
\newcommand{\be}{{\beta}}
\newcommand{\de}{{\delta}}
\newcommand{\ga}{{\gamma}}
\newcommand{\si}{{\sigma}}
\newcommand{\la}{{\lambda}}
\newcommand{\te}{{\theta}}

\newcommand{\Si}{{\Sigma}}

\newcommand{\calA}{{\mathcal A}}

\newcommand{\F}{{\mathcal A}}
\newcommand{\G}{{\mathcal B}}
\newcommand{\cK}{{\mathcal K}}
\newcommand{\cH}{{\mathcal H}}
\newcommand{\cL}{{\mathcal L}}
\newcommand{\cT}{{\mathcal T}}
\newcommand{\cF}{{\mathcal F}}

\newcommand{\tG}{{ G_{0}}}
\newcommand{\tF}{{ F_0}}

\newcommand{\gab}{{\jcov{A,B}}}
\newcommand{\jcov}[1]{g_{#1}}
\newcommand{\gabp}{{\jcov{A',B'}}}
\newcommand{\fz}{{G^0}}
\newcommand{\Sz}{{\Sigma}}
\newcommand{\Sp}{\Sigma_+}
\newcommand{\Sm}{\Sigma_-}
\newcommand{\tee}{{p_\ee}}
\newcommand{\un}{{1}}
\newcommand{\Pg}{{P'_{G^0}}}
\newcommand{\ang}[2]{{\{#1\leq\theta\leq#2\}}}
\newcommand{\tdue}{{ y}}
\newcommand{\se}{{ S}}

\newcommand{\Az}{{A^0}}

\begin{document}
\title[The covariogram  determines convex $3$-polytopes]{The covariogram
determines  three-dimensional convex polytopes}
\author{Gabriele Bianchi}
\address{Dipartimento di Matematica, Universit\`a di Firenze, 
Viale Morgagni 67/A, Firenze, Italy I-50134}
\email{gabriele.bianchi@unifi.it}
\subjclass[2000]{Primary 60D05; Secondary 52A22, 42B10, 52B10, 52A38}
\keywords{Autocorrelation, covariogram, cross covariogram, cut-and-project scheme, geometric tomography, phase retrieval, quasicrystal, set covariance}
\date{\today}

\begin{abstract}The cross covariogram $g_{K,L}$ of two convex sets
$K$, $L\subset\Real^n$ is the function which associates to each $x\in
\Real^n$ the volume of the intersection of $K$ with $L+x$. The problem of determining the 
sets from their covariogram is relevant in stochastic geometry, in probability and it is equivalent to a particular case of the phase retrieval problem in Fourier analysis. It is also
relevant for the inverse problem of determining the atomic structure of a quasicrystal from its X-ray
diffraction image.
The two main results of this
paper are that $g_{K,K}$ determines three-dimensional convex polytopes $K$ and that $g_{K,L}$ determines both  $K$ and $L$ when $K$ and $L$ are convex polyhedral cones satisfying certain assumptions. 
These results settle a conjecture of G.~Matheron in the class 
of convex polytopes.
Further results regard the known counterexamples in dimension 
$n\geq4$.
We also introduce and study the
notion of synisothetic polytopes.
This concept is related to the rearrangement of the
faces of a convex polytope.
\end{abstract}
\maketitle

\section{Introduction}
Let $K$ be a convex body in $\Real^n$. The \emph{covariogram} $g_K$
of $K$ is the function
\[
g_K(x)= \voln(K\cap(K+x)),
\]
where $x\in\Real^n$ and $\voln$ denotes $n$-dimensional Lebesgue measure.
This functional, which was introduced by Matheron in his book
\cite{M} on random sets, is also sometimes called the \emph{set
covariance} and it coincides with the  \emph{autocorrelation} of the characteristic function of $K$:
\begin{equation}\label{convoluzione}
g_K=1_K\ast 1_{(-K)}.
\end{equation}
The covariogram $g_K$ is clearly unchanged by a translation or a reflection 
of $K$. (The term \emph{reflection} will always mean reflection 
in a point.) Matheron~\cite{M2} in 1986 asked the following
question and conjectured a positive answer for the case $n=2$.

\smallskip

\textbf{Covariogram problem.} \emph{Does the covariogram determine a
convex body, among all convex bodies, up to translations and
reflections?}

\smallskip

The conjecture regarding $n=2$ has been completely settled only very recently, by Averkov and Bianchi~\cite{AB2}. 

Matheron~\cite[p.~86]{M} observed  that, for  $u\in S^{n-1}$ and for
all $r>0$, the derivatives $({\pa}g_K/{\pa r})(ru)$ give the
distribution of the lengths of the chords of $K$ parallel to $u$.
Such information is common in stereology, statistical shape
recognition and image analysis, when properties of an unknown body
are to be inferred from chord length measurements; see~\cite{S},
\cite{CB} and~\cite{Se}, for example. 
Blaschke asked whether the distribution of
the lengths of chords (in all directions) of a convex body characterizes the body, but Mallows and Clark~\cite{MC}
proved that this is false even for convex polygons.  In fact (see~\cite{N1}) the covariogram problem is equivalent to the problem of
determining a convex body from all its separate chord length
distributions, one for each direction $u\in S^{n-1}$.

The covariogram problem appears in other contexts. 
Adler and Pyke~\cite{AP1} asked in 1991 whether the
distribution of the difference $X-Y$ of independent random variables
$X$ and $Y$ uniformly distributed over $K$ determines $K$, up to
translations and reflections. Since the convolution in \eqref{convoluzione} 
is, up to a multiplicative factor, the
probability density of $X-Y$, this problem is equivalent to the
covariogram problem. The same authors~\cite{AP2} find the
covariogram problem relevant also in the study of scanning Brownian
processes and of the equivalence of measures induced by these
processes for different base sets.

The covariogram problem is a special case of the {\em
phase retrieval problem} in Fourier analysis. This problem involves determining a function $f$ from the modulus of its Fourier transform $\widehat{f}$. 
Since phase and amplitude are in general independent of each other,
in order to solve the phase retrieval problem one must use
additional information constraining the admissible solutions $f$;
see~\cite{KST} and~\cite{Sanz}.
Taking Fourier transforms in \eqref{convoluzione} and using the
relation $\widehat{1_{-K}}=\overline{\widehat{1_K}}$, we obtain
\begin{equation} \label{e:fourier}
\widehat{g_K}=\widehat{1_K}\widehat{1_{-K}} =|\widehat{1_K}|^2.
\end{equation}
Thus the phase retrieval problem, restricted to the class of
characteristic functions of convex bodies, reduces to the
covariogram problem.

In X-ray crystallography the atomic structure of a crystal is to be
found from diffraction images. As Rosenblatt~\cite{Ros}
explains, ``Here the phase retrieval problem arises because the
modulus of a Fourier transform is all that can usually be measured
after diffraction occurs.'' 
A convenient way of describing many important examples of quasicrystals is via the ``cut and project scheme''; see~\cite{BaakeMoody}. Here to the atomic structure, represented by a discrete set $S$ contained in a space $E$, is associated a lattice $N$ in a higher dimensional space $E\times E'$ and a ``window'' $W\subset E'$ (which in many cases is  a convex set). Then  $S$ coincides with the projection on $E$ of the points of the lattice $N$ which belong to $W \times E$. In many examples the lattice $N$ can be determined by the diffraction image. To determine $S$ it is however necessary to know $W$: the covariogram problem enters at this point, since the covariogram of $W$ can be obtained by the diffraction image; see~\cite{BaakeGrimm}.

In convex geometry the covariogram  appears in several contexts. 
For instance it has a central role in the proof of the Rogers-Shephard inequality~\cite[Th.~7.3.1]{Sc}. 
Moreover the level sets of $g_K$, which are convex and are called 
\emph{convolution bodies}, have been studied~\cite{Tso} and are 
related to the projection body of $K$.
A discrete version of the covariogram problem has been considered~\cite{GGZ}. 
In~\cite{GZ}  the covariogram problem was transformed to a question for the \emph{radial mean bodies}. 

The first contribution to Matheron's question was made  by Nagel~\cite{N1} in 1993, who confirmed Matheron's  conjecture for all convex polygons. Other partial results towards the complete confirmation of this conjecture in the plane have been proved by Schmitt~\cite{S}, Bianchi, Segala and Vol\v{c}i\v{c}~\cite{BSV}, Bianchi~\cite{B2} and Averkov and Bianchi~\cite{AB}.

In general, Matheron's conjecture is false, as the author~\cite{B2}
proved by finding counterexamples in $\Real^n$, for any $n\geq 4$. Indeed,
the covariogram of the Cartesian product of convex sets  $K\subset
\Real^k$ and $L\subset \Real^m$ is the product of the covariograms
of $K$ and $L$. Thus $K\times L$ and $K\times(-L)$ have equal
covariograms. However, if neither $K$ nor $L$ is centrally symmetric,
then $K\times L$ cannot be obtained from $K\times (-L)$ through a
translation or a reflection. To satisfy these requirements the
dimension of both sets must be at least two and thus the dimension
of the counterexamples is at least four. 
We note that these counterexamples can be polytopes but
not $C^1$ bodies.

For $n$-dimensional convex polytopes $P$, Goodey, Schneider and Weil
\cite{GSW2} prove that if $P$ is simplicial and  
$P$ and $-P$ are in general relative position, 
the covariogram
determines $P$. Up till now this was the only positive result
available in dimension $n\geq 3$ (beside the positive result for all
centrally symmetric convex bodies).

In  $\Real^3$ the ``Cartesian product'' construction does not apply
because any one-dimensional convex set is centrally symmetric. For
the class of three-dimensional convex polytopes we are able to confirm
Matheron's conjecture. This answer, together with the
counterexamples for $n\geq4$, completely settles the covariogram
problem for convex polytopes.

\begin{theorem}\label{teorema0}Let $P\subset\Real^3$ be  a convex polytope
with non-empty interior. Then $g_P$ determines $P$, in the class of
convex bodies in $\Real ^3$, up to translations and reflections.
\end{theorem}

Given a face $F$ of a convex polytope $P\subset\Real^3$,
$\cone(P,F)$  denotes the support cone to $P$ at $F$ (see the next section for all
unexplained definitions).  If $w\in S^2$ we denote by $P_w$ the unique proper  face of $P$ such that the relative interior of its normal cone contains  $w$. 
A major step in the proof of Theorem~\ref{teorema0} is the following result.

\begin{theorem}\label{teorema_lmr}Let $P$ and $P'$ be convex 
polytopes in $\Real^3$ with non-empty interior such that $g_P=g_{P'}$.
If $w\in S^2$ then, possibly after a translation or a reflection of
$P'$ that may depend on $w$,
\begin{align}
P_w&=P'_w;\label{t_lmr_facce}\\
\cone(P,P_w)&=\cone(P',P'_w).\label{t_lmr_coni}
\end{align}
\end{theorem}

Let $P$ and $Q$ be convex polytopes in $\Real^n$, let $F$ be a
proper face of $P$, and let $G$ be a proper face of $Q$. We say that
$F$ and $G$ are {\em isothetic} if $G$ is a translate of $F$ and
\[
\cone(P,F)=\cone(Q,G).
\]
Given convex polytopes $P_1$, $P_2$, $Q_1$ and $Q_2$ in $\Real^n$ we
say that $(P_1,P_2)$ and $(Q_1,Q_2)$ are \emph{synisothetic} if
given any proper face $F$ of $P_j$, for some $j=1,2$, there is a
proper face $G$  of $Q_k$, for some $k=1,2$ (and conversely), such
that $F$ and $G$ are isothetic.

The term synisothetic was suggested by P.~McMullen. The previous
theorem can be rephrased in these terms:  If $g_P=g_{P'}$, then
$(P,-P)$ and $(P',-P')$ are synisothetic. 

In order to prove Theorem~\ref{teorema_lmr} we  investigate  two
related problems. The presence of parallel facets of $P$  causes
difficulties (eliminated by the special assumption in~\cite{GSW2}
that $P$ and $-P$ are in general relative position). To deal with this,
A.~Vol\v{c}i\v{c} and R.~J.~Gardner posed a generalization of Matheron's
question we call the \emph{cross covariogram problem}.  To explain
this, some terminology is needed. Given two convex sets $K$ and $L$
in $\Real ^n$, the {\em cross covariogram} $\jcov{K,L}$ is the
function
\[
\jcov{K,L}(x)=\voln(K\cap (L+x))
\]
where $x\in\Real^n$ is such that $\voln(K\cap (L+x))$ is finite. 

Let $K$, $L$, $K'$ and $L'$ be  convex sets in $\Real^n$. 
We call $(K,L)$ and $(K',L')$ \emph{trivial associates} if one pair is obtained by the other one via a combination of the operations which leave the cross covariogram unchanged; see Section~\ref{definitions} for the precise definition.

\begin{problem}[Cross covariogram problem for polygons]\label{jcpf}\emph{Does the cross covariogram of the convex polygons $K$ and $L$ determine the pair $(K,L)$, among all pairs of convex bodies, up to trivial associates?}
\end{problem}

One connection between covariogram and cross covariogram lies in the observation that if $F$ and $G$ denote parallel facets of a convex polytope $P\subset\Real^n$, 
the ``singular part'' of some second order distributional derivative of $g_P$ provides both $\jcov{F,G_0}$ and $g_F+g_\tG$, where $\tG$ is the orthogonal projection of $G$ on the hyperplane which contains $F$. We prove that the information given by these two functions can be decoupled and provides both $g_F$ and $g_\tG$, up to an exchange between them. 
When $n=3$, in view of the confirmation of Matheron's conjecture in the plane,  $g_P$ provides both $F$ and $\tG$, up to an exchange between them and up to translations or reflections of $F$ and of $\tG$. 
However, all this is not sufficient for our purpose and a detailed study of Problem~\ref{jcpf} is needed; see Remarks~\ref{rem_facce} and~\ref{rem_facce_four} for further comments.

The answer to Problem~\ref{jcpf} is negative as Examples~\ref{parall} and~\ref{parall_due} show (see Figures~\ref{fig_parall} and~\ref{fig_parall_due}). For each choice of some real parameters there exist four  pairs of
parallelograms $(\cK_1,\cL_1),\dots,(\cK_4,\cL_4)$ such that, for $i=1,3$, $g_{\cK_i,\cL_i}=g_{\cK_{i+1},\cL_{i+1}}$ but $(\cK_i,\cL_i)$ is not a trivial associate of $(\cK_{i+1},\cL_{i+1})$. 
Problem~\ref{jcpf} is completely solved by Bianchi~\cite{B4}, which proves that, up to an affine transformation, the previous counterexamples are the only ones.
\begin{theorem}[\cite{B4}]\label{cov_congiunto_poligoni}
Let $K, L$ be convex polygons and $K',L'$ be planar convex bodies
with $g_{K,L}=g_{K',L'}$. Assume that there is no affine transformation $\cT$ and no different  indices $i,j$, with either $i,j\in\{1,2\}$ or  $i,j\in\{3,4\}$, such that $(\cT K,\cT L)$ and $(\cT K',\cT L')$ are trivial associates of $(\cK_i,\cL_i)$ and $(\cK_j,\cL_j)$, respectively.
Then  $(K,L)$ is a trivial associate of $(K',L')$.
\end{theorem}
This result  states that the information provided by the cross covariogram of convex polygons is so rich as to  determine not only one unknown body, as required by Matheron's conjecture, but two bodies, with a few exceptions.

The second problem is in some sense dual to the first one and has
been introduced by Mani-Levitska~\cite{ML}. Let $A$ and $B$ be
convex polyhedral cones  in $\Real^3$, with apex the origin $O$ and $A\cap
B=\set{O}$.
\begin{problem}[Cross covariogram problem for cones]\label{jcpc}\emph{Does the cross covariogram of $A$ and $B$ determine the pair $(A,B)$, among all pairs of convex cones, up to trivial associates?}
\end{problem}
Proposition~\ref{proposizione_coni} provides an answer to
Problem~\ref{jcpc}, while Bianchi~\cite{B4} (see Lemma~\ref{cov_cong_settori_angolari} in this paper) solves
completely the corresponding problem  in the plane, 
describing also some situations of non-unique determination. The techniques that
we use to prove Proposition~\ref{proposizione_coni} 
rely on two main ingredients. The first one
is an analysis of  the set where $g_{A,B}$ is not $C^3$. 
The second ingredient is the observation that a suitable second
order mixed derivative of $g_{A,B}$ provides certain $X$-ray 
functions of the cones. Some results regarding the determination of
convex polyhedral cones from their $X$-ray functions are also needed (see~\cite{B3}).

The proof of Theorem~\ref{teorema0} also requires  the study of the structure of $\pa P\cap\pa P'$, when  $P$ and $P'$ are convex polytopes in $\Real^3$ and $(P,-P)$ and $(P',-P')$ are synisothetic. This study is contained in Section~\ref{sec_syn} while Theorem~\ref{teorema0} is proved in Section~\ref{sec_mainproof}.

In Section~\ref{nmaggiore} the counterexamples in dimension $n\geq 4$ are presented in terms of decomposition of a convex body into direct summands, and the relation between the covariogram and this decomposition is studied. Theorem~\ref{prodotti_cartesiani} classifies the convex bodies which have covariogram equal to one of the counterexamples.
In view of all this the right form to ask Matheron's problem for $n$-dimensional convex  polytopes $P$, when $n\geq4$, is with the restriction to  \emph{directly indecomposable} $P$.

\textsc{Acknowledgements.} We are extremely grateful to G.~Averkov, R.~J.~Gardner and P.~Gronchi for reading large parts of this manuscript and suggesting many arguments that simplified and clarified some of the proofs. We also thank A.~Zastrow, for a very useful discussion regarding Lemma~\ref{jordaninverse}, and P.~Mani-Levitska, for giving us his unpublished note~\cite{ML}.


\section{Definitions, notations and preliminaries}\label{definitions}
For convenience of the reader, we repeat here all the definitions already introduced.
As usual, $S^{n-1}$ denotes the unit sphere in $\Real^n$, centered at 
the origin $O$.  For $x$, $y\in \Real^n$,  $\|x\|$ is the Euclidean norm of $x$ and
$x\cdot y$ denotes scalar product. For $\de>0$, $B(x,\de)$ denotes the open ball in $\Real^n$ centered at $x$ and with radius $\de$. If $u\in S^{n-1}$,
$u^{\perp }$ denotes the $(n-1)$-dimensional subspace  orthogonal to $u$.

If $A\subset \Real^n$  we denote by $\inte A$, $\cl A$, 
$\partial A$ and
$\conv A$ the \emph{interior}, \emph{closure}, \emph{boundary} and
\emph{convex hull} of $A$, respectively. The {\em characteristic function} of $A$
is denoted by $1_A$. The reflection of $A$ in the origin is $-A$. With the symbol $A\mathbin|\pi$ we denote the orthogonal projection of $A$ on the affine space  $\pi$. 
Moreover we define $\pos A=\set{\mu x : x\in A,\: \mu\geq 0}$.
The symbol $\voln$  denotes $n$-dimensional Lebesgue measure, while $\cH^k$ denotes $k$-dimensional Hausdorff measure in $\Real^n$, where $0\leq k\leq n$.

\textit{Convex sets.} 
A {\it convex body} $K\subset\Real^n$ is a compact convex set with non-empty interior.
The symbol $\aff K$ stands for the \emph{affine hull} of $K$; $\dim K$ is the dimension of $\aff K $. The symbols $\relbd K$ and $\relint K$ indicate respectively the
\emph{relative boundary} and the \emph{relative interior} of $K$. The
\emph{difference body} of $K$ is defined by
\[
DK=K+(-K)=\{x-y: x,~y\in K\}.
\]
The \emph{support function} of $K$ is defined, for $x\in\Real^n$, by 
$
h_K(x)=\sup \set{x\cdot y : y\in K}.
$
The Steiner point of $K$ is defined by $s(K)=(1/\la_n(B(0,1)))\int_{S^{n-1}} h_K(u)u\:d\cH^{n-1}(u)$. 
We write 
\begin{equation}\label{decomp_somma_diretta}
K=K_1\oplus\dots\oplus K_s
\end{equation} 
if $K=K_1+\cdots+K_s$ for suitable convex bodies $K_i$ lying in  linear subspaces 
$E_i$ of $\Real^n$ such that $E_1\oplus\cdots\oplus E_s=\Real^n$.  
If a representation $K=L\oplus M$ is only possible with 
$\dim L=0$ or $\dim M=0$ then $K$ is \emph{directly indecomposable}. Each $K$, with $\dim K\geq 1$, has a representation as in \eqref{decomp_somma_diretta}, with $\dim K_i\geq 1$ and $K_i$ directly indecomposable, which is unique up to the order of the summands.

Given $x,y\in\Real^n$, we write $[x,y]$ for the segment with endpoints $x$ and $y$. Given a convex body $K\subset\Real^2$ and $a,b\in\pa K$ the symbol $[a,b]_{\pa K}$  denotes $\set{p\in \pa K :a\leq p\leq b}$  in counterclockwise order on $\pa K$ and $(a,b)_{\pa K}$ denotes the corresponding open arc. 
We will refer to $a$ as the \emph{lower endpoint} of the arc and to $b$ as its \emph{upper endpoint}. 

The \emph{X-ray} of a convex set
$K\subset\Real^n$ with respect to $u\in S^{n-1}$ is the function which
associates to each line $l$ parallel to $u$  the length of $K\cap
l$. The \emph{$-1$-chord function} of $K$ at $p\in \Real^n\setminus\cl{K}$ is defined, 
for each line $l$ through $p$, by
\[
\int_{K\cap l}\|x-p\|^{-2}d\len(x).
\]

\textit{Polytopes.} Let $P$  be a convex polytope in $\Real^n$. 
As usual the $0$-, $1$- and $(n-1)$-dimensional faces are called vertices, edges and facets, respectively.
Given a face  $F$  of $P$ the {\em normal cone} of $P$ at $F$ is denoted by
$N(P,F)$ and is the set of all outer normal vectors to $P$ at $x$,
where $x\in\relint F$, together with $O$. The {\em support cone of $P$ at $F$} is  the
set
\[
\cone(P,F)=\set{\mu(y-x)\;:\;y\in P\;,\;\mu\geq0},
\]
where $x\in \relint F$. Neither definitions  depend on the choice
of $x$.  Two faces $F$ and $G$ of  $P$ are \emph{antipodal} if $\relint
N(P,F)\cap(- \relint N(P,G))\neq\emptyset$.
Given $u\in S^{n-1}$ the \emph{exposed face of $P$ in direction $u$} is 
\[
P_u=\{x\in P : x\cdot u=h_P(u)\}.
\]
It is the unique proper face of $P$ such that the relative interior of its normal cone contains $u$. 
We will repeatedly use the following identities, proved in~\cite[Th.~1.7.5(c)]{Sc} and valid for all $u\in S^{n-1}$ and all convex polytopes $P$, $P'$ in $\Real^n$:
\begin{equation}\label{facce_corpodiff}
(P+P')_u=P_u+P'_u;\quad (D\,P)_u=P_u+(-P)_u.
\end{equation}
Given a face $G$ of $P$,  the meaning of $\Sp(G)$, $x^+(G)$, $x^-(G)$, $P_G$ and of positive, negative or neutral face is introduced in Definitions~\ref{def_positive} and~\ref{definizione_Sp}. See the statement of Lemma~\ref{sigma_meno} for the meaning of  $\Sm(G)$. We say that $P$ and $-P$ are in \emph{general relative position} if $\dim P_w\cap (P_{-w}+x)=0$ for each $w\in S^{n-1}$ and for each $x\in \Real^n$.

In this paper the term \emph{cone} always means cone with apex $O$. A
polyhedral convex cone is \emph{dihedral} if it is the intersection
of two closed half-spaces. 
If $F$ is an edge of a three-dimensional convex polytope $P$, then $\cone(P,F)$ is dihedral. A convex cone is \emph{pointed} if its apex is a vertex.

\textit{Synisothesis.} 
Let $P$ and $Q$ be convex polytopes in $\Real^n$, let $F$ be a
proper face of $P$, and let $G$ be a proper face of $Q$. We say that
$F$ and $G$ are {\em isothetic} if $G$ is a translate of $F$ and
\[
\cone(P,F)=\cone(Q,G).
\]
Given convex polytopes $P_1$, $P_2$, $Q_1$ and $Q_2$ in $\Real^n$ we
say that $(P_1,P_2)$ and $(Q_1,Q_2)$ are \emph{synisothetic} if
given any proper face $F$ of $P_1$ or of $P_2$ there is a
proper face $G$  of $Q_1$ or of $Q_2$ (and conversely) such
 that $F$ and $G$ are isothetic.

\textit{Covariogram and trivial associates.} Let $K$, $L$, $K'$ and $L'$ be  convex sets in $\Real^n$. 
The {\em cross covariogram} $g_{K,L}$ is the function $g_{K,L}(x)=\voln(K\cap (L+x))$, where $x\in\Real^n$ is such that $\voln(K\cap (L+x))$ is finite. 
It is evident that $g_{K,K}=g_K$ and that, for any $x\in\Real^n$,
\[
g_{K,L}=g_{K+x,L+x}=g_{-L,-K}.
\]
We call $(K,L)$ and $(K',L')$  \emph{trivial associates} if either $(K,L)=(K'+x,L'+x)$ or 
$(K,L)=(-L'+x,-K'+x)$, for some $x\in \Real^n$. 
It is easy to prove that
\begin{align}
&\supp g_{K}=D\,K,& &\text{ }& &\supp g_{K,L}=K+(-L),\label{support}\\
&g_K=1_K\ast 1_{-K},& &\text{ }& &g_{K,L}=1_K\ast 1_{-L},\label{convolution}\\
&\widehat{g_K}=|\widehat{1_K}|^2& &\text{and}&
&\widehat{g_{K,L}}=\widehat{1_K}\cdot
\overline{\widehat{1_L}},\label{fourier}
\end{align}
where $\supp f$ and $\widehat{f}$ denote respectively the support
and the Fourier transform of the function $f$.

\section{The cross covariogram problem for polygons}

This section recalls some results proved in~\cite{B4} and needed in this paper.
Let  $(\rho,\te)$ denote polar coordinates.
For brevity, given $\al,\be\in[0,2\pi]$ with $\al<\be$, we write  $\ang{\al}{\be}$ for the  cone $\{(\rho,\te):\,\al \leq\te\leq\be\}$.

\begin{example}
Let $\F_1=\ang{0}{3\pi/4}$, $\G_1=-\ang{{\pi}/4}{{\pi}/2}$,
$\F_2=\ang{0}{{\pi}/{4}}$ and  $\G_2=-\ang{{\pi}/2}{{3\pi}/4}$.
We have $\set{\F_1,-\G_1}\neq\set{\F_2,-\G_2}$ and $\jcov{\cT\F_1,\cT\G_1}=\jcov{\cT\F_2,\cT\G_2}$, for any non-singular affine transformation $\cT$. 
\end{example}

\begin{lemma}\label{cov_cong_settori_angolari}
Let $A$, $B$, $A'$ and $B'$ be pointed closed convex cones in $\Real^2$ with non-empty interior and apex the origin $O$ such that $\inte A\cap \inte B=\emptyset$.
The identity $\jcov{A,B}=\jcov{A',B'}$ holds if and only if one of the following cases occurs:
\begin{enumerate}
\item\label{acaso0} $\set{A,-B}=\set{A',-B'}$;
\item\label{acaso12} there exist a linear transformation $\cT$ and $i,j\in\set{1,2}$, $i\neq j$, such that
\begin{equation}\label{acaso1}
\{\cT A,-\cT B\}=\{\F_i,-\G_i\}\quad\text{and}\quad \{\cT A',-\cT B'\}=\{\F_j,-\G_j\}.
\end{equation}
\end{enumerate}
\end{lemma}

\begin{remark}\label{vettore_in_comune}Observe that $\inte\F_1\cap\inte(-\G_1)\neq\emptyset$ and $\inte\F_2\cap\inte(-\G_2)=\emptyset$. Thus if 
$\inte A\cap\inte(-B)$ and  $\inte A'\cap\inte(-B')$ are both empty or  both non-empty, then Lemma~\ref{cov_cong_settori_angolari} implies $\set{A,-B}=\set{A',-B'}$. 
\end{remark}

\begin{example}\label{parall}
Let $\al,\be,\ga$ and $\de$ be positive real numbers, let $y\in\Real^2$ and  let $I_1=[(-1,0),(1,0)]$, $I_2=(1/\sqrt{2})\ [(-1,-1),(1,1)]$, $I_3=[(0,-1),(0,1)]$ and $I_4=(1/\sqrt{2})\ [(1,-1),(-1,1)]$. We define four parallelograms as follows:
$\cK_1=\al I_1+\be I_2$; $\cL_1=\ga I_3+\de I_4+y$;
$\cK_2=\al I_1+\de I_4$ and $\cL_2=\be I_2+\ga I_3+y$.
See Fig.~\ref{fig_parall}. 

The pairs $(\cK_1,-\cL_1)$, $(\cK_2,-\cL_2)$ are not synisothetic (no vertex of a polygon in the second pair has a support cone equal to the support cone of the top vertex of $\cL_1$ or to its reflection). Moreover $\jcov{\cK_1,\cL_1}=\jcov{\cK_2,\cL_2}$.
\begin{figure}
\begin{center}
\input{parallelogrammi_5.pstex_t}
\end{center}
\caption{Up to affine transformations, $(\cK_1,\cL_1)$ and $(\cK_2,\cL_2)$ are the only pairs of convex polygons with equal cross covariogram which are not synisothetic. }
\label{fig_parall}
\end{figure}
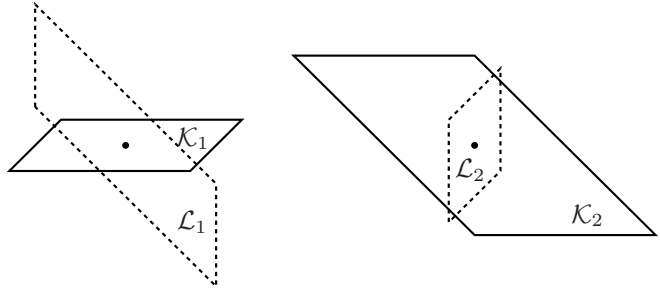
\end{example}

\begin{example}\label{parall_due}Let $\al,\be,\ga$ and $\de$ be positive real numbers, let $m\in \Real$, $y\in\Real^2$ and let  $I^{(m)}=(1/\sqrt{1+m^2})\,[(-m,-1),(m,1)]$. Assume either  $m=0$, $\al\neq \ga$ and $\be\neq\de$ or else $m\neq 0$ and $\al\neq \ga$. We define four parallelograms as follows:
$\cK_3=\al I_1+\be I_3$;
$\cL_3=\ga I_1+\de I^{(m)}+y$;
$\cK_4=\ga I_1+\be I_3$ and
$\cL_4=\al I_1+\de I^{(m)}+y$.
See Fig.~\ref{fig_parall_due}. 

We have $\jcov{\cK_3,\cL_3}=\jcov{\cK_4,\cL_4}$  and the pairs $(\cK_3,-\cL_3)$ and $(\cK_4,-\cL_4)$ are synisothetic. 
However, $(\cK_3,\cL_3)$ and $(\cK_4,\cL_4)$ are not trivial associates. 
\begin{figure}
\begin{center}
\input{parallelogrammi_tre.pstex_t}
\end{center}
\caption{Up to affine transformations, $(\cK_3,\cL_3)$ and $(\cK_4,\cL_4)$ are the only pairs of convex polygons with equal cross covariogram which are synisothetic and are not trivial associates.}
\label{fig_parall_due}
\end{figure}
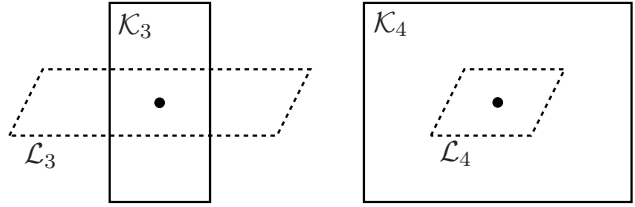
\end{example}

\begin{lemma}\label{coni_finale_shortened}
Let $A$, $B$, $C$ and $D$ be convex cones in  $\Real^n$ with apex the origin $O$. Assume that each of them either coincides with $\set{O}$ or has non-empty interior, and, moreover,
$A\cup B\subset \set{(x_1,x_2,\dots,x_n):x_n\geq0}$,
$A\cap\set{x_n=0}=B\cap\set{x_n=0}=\set{O}$,
$C\cup D\subset \set{x_n\leq0}$ and
$\conv (C\cup D)$ is pointed.
If $g_{A,C}+g_{B,D}=g_{A,D}+g_{B,C}$
then either $A=B$ or  $C=D$.
\end{lemma}

\section{Determining the faces of $P$: proof  of  \eqref{t_lmr_facce} in Theorem~\ref{teorema_lmr} }
The result regarding $g_\tF+g_\tG$ in next proposition was first observed by K.~Rufibach~\cite[p.14]{R}.
\begin{proposition}\label{rufibach} Let $P\subset\Real^n$ be a convex polytope with non-empty interior, let $w\in S^{n-1}$, $F=P_w$ and $G=P_{-w}$. The covariogram $g_P$ determines both $g_\tF+g_\tG$ and $g_{\tF,\tG}$, where  $\tF=F\mathbin| w^\perp$ and $\tG=G\mathbin| w^\perp$.
\end{proposition}

The possibility of proving Proposition~\ref{rufibach} using the expression of the second order distributional derivative of $g_P$ computed in next lemma was  suggested by G.~Averkov. Let $C^\infty_0(\Real^n)$ denote the class of infinitely differentiable functions on $\Real^n$ with compact support.
\begin{lemma}\label{deriv_distrib}
Let $P\subset\Real^n$ be a convex polytope with non-empty interior. Let $F_1,\dots, F_m$ be its facets, $\nu_i$ be the unit outer normal of $P$ at $F_i$, for $i=1,\dots,m$, let $w\in S^{n-1}$ and let $I_p=\{(i,j):\text{$F_i$ is parallel to  $F_j$}\}$ and $I_{np}=\{(i,j):\text{$F_i$ is not parallel to  $F_j$}\}$. Then, for $\phi\in C^\infty_0(\Real^n)$, we have
\begin{multline}\label{second_derivatives}
-\frac{\pa^2 g_P}{\pa w^2} (\phi)=
\sum_{(i,j)\in I_{np}} \frac{w\cdot\nu_i\ w\cdot\nu_j}{\sqrt{1-(\nu_i\cdot\nu_j)^2}}
\int_{\Real^n}\cH^{n-2} (F_i\cap(F_j+z))\,\phi(z)\,dz\\
+\sum_{(i,j)\in I_{p}} w\cdot\nu_i\ w\cdot\nu_j
\int_{F_i-F_j}\cH^{n-1}(F_i\cap(F_j+z))\,\phi(z)\,d\cH^{n-1}(z).
\end{multline}
Both sums in the right hand side of \eqref{second_derivatives} are uniquely determined by $g_P$.
\end{lemma}
\begin{proof} 
Let $\de_{F_i} (\phi)=\int_{F_i}\phi(x)d\cH^{n-1}(x)$. 
It is easy to prove that  $(\pa 1_P / \pa w) (\phi)=-\sum_{i=1}^m w\cdot\nu_i\de_{F_i} (\phi)$. For instance,~\cite[p.60]{H} proves the corresponding formula for sets with $C^1$ boundary  and the formula for $P$ can be proved by an approximation argument. If $(P_n)$ is a sequence of convex bodies with $C^1$ boundary converging to $P$ in the Hausdorff metric, then  $(\pa 1_P / \pa w) (\phi)=\lim_n (\pa 1_{P_n} / \pa w) (\phi)$, by  dominated convergence Theorem~\cite[p.~20]{EG}. Thus, if $\nu_{P_n}$ denotes the outer normal to $\pa P_n$, we have 
\begin{align*}
\frac{\pa 1_P}{\pa w} (\phi)=\lim_n \frac{\pa 1_{P_n}}{\pa w} (\phi)
&=-\lim_n\int_{\pa P_n} w \cdot\nu_{P_n}(x)\ \phi(x)\  d\cH^{n-1}(x)\\
&=-\sum_{i=1}^m w\cdot\nu_i\ \de_{F_i} (\phi).
\end{align*}

Since $\pa 1_P / \pa w$ has compact support and $g_P=1_P\ast1_{-P}$ we can write
\begin{equation}\label{somma_dirac}
\begin{aligned}
\frac{\pa^2 g_P}{\pa w^2} (\phi)&=
\left(\frac{\pa1_P}{\pa w}\ast\frac{\pa1_{-P}}{\pa w}\right) (\phi)\\
&=-\sum_{i,j=1}^m w\cdot\nu_i\ w\cdot\nu_j\  (\de_{F_i}\ast\de_{-F_j}) (\phi).
\end{aligned}
\end{equation}
Assume that $F_i$ and $F_j$ are parallel and choose a Cartesian coordinates system so that $F_i\subset\{x\in\Real^n:x_2=0\}$ and $F_j\subset\{x:x_2=\al\}$, where $x=(x_1,x_2)\in\Real^{n-1}\times\Real$ and $\al\in\Real$. We have
\begin{align*}
(\de_{F_i}\ast\de_{-F_j}) (\phi)&=
\int_{F_i}\left(
\int_{-F_j}\phi(x+y) d\cH^{n-1}(y)
\right) d\cH^{n-1}(x)\\
&=\int_{\Real^{n-1}}1_{F_i}(x_1,0)\left(
\int_{\Real^{n-1}}1_{-F_j}(y_1,-\al) \phi(x_1+y_1,-\al) dy_1
\right) dx_1\\
&=\int_{\Real^{n-1}}\left(
\int_{\Real^{n-1}}1_{F_i}(x_1,0)1_{-F_j}(z_1-x_1,-\al) dx_1
\right)\phi(z_1,-\al) dz_1.\\
\end{align*}
Since $1_{F_i}(x_1,0)1_{-F_j}(z_1-x_1,-\al)=1_{F_i\cap(F_j+(z_1,-\al))}(x_1,0)$ and $F_i-F_j\subset\{x: x_2=-\al\}$, we have
\begin{equation}\label{dist_facce_parallele}
\begin{aligned}
(\de_{F_i}\ast\de_{-F_j}) (\phi)&=\int_{\Real^{n-2}}\cH^{n-1}(F_i\cap(F_j+(z_1,-\al)))\phi(z_1,-\al) dz_1\\
&=\int_{F_i-F_j}\cH^{n-1}(F_i\cap(F_j+z))\,\phi(z)\,d\cH^{n-1}(z).
\end{aligned}
\end{equation}

Assume $n\geq3$, that $F_i$ and $F_j$  are not parallel and choose a Cartesian coordinates system so that $F_i\subset\{x\in\Real^n:x_3=0\}$ and $F_j\subset\{x:x_1=\al x_3\}$, where $x=(x_1,x_2,x_3)\in\Real\times\Real^{n-2}\times\Real$ 
and $\al\in\Real$. We have
\begin{align*}
(\de_{F_i}\ast&\de_{-F_j}) (\phi)=\sqrt{1+\al^2}\int_{\Real\times\Real^{n-2}}1_{F_i}(x_1,x_2,0)\\
&\quad\left(
\int_{\Real^{n-2}\times\Real}1_{-F_j}(\al y_3,y_2,y_3)\,\phi(x_1+\al y_3,x_2+y_2,y_3)\, dy_2\,dy_3
\right)\, dx_1\,dx_2\\
&=\sqrt{1+\al^2}\int_{\Real^n}\left(
\int_{\Real^{n-2}}1_{F_i}(z_1-\al z_3,x_2,0)\,1_{-F_j}(\al z_3,z_2-x_2,z_3)\,dx_2
\right)\phi(z) dz,\\
\end{align*}
where $z$ denotes $(z_1,z_2,z_3)$ and in the last integral we have used the change of variable $(x_1,y_2,y_3)=(z_1+\al z_3,z_2-x_2,z_3)$.
Since $1_{F_i} (z_1-\al z_3,x_2,0) 1_{-F_j} (\al z_3,z_2-x_2,z_3) =1_{F_i\cap(F_j+z)} (z_1-\al z_3,x_2,0) $ and $F_i\cap(F_j+z)\subset\{x: x_1=z_1-\al z_3, x_3=0\}$, the inner integral in the last line of the previous formula equals $\cH^{n-2}(F_i\cap(F_j+z))$. Thus, since $\sqrt{1+\al^2}=(1-(\nu_i\cdot \nu_j)^2)^{-1/2}$,  we have
\begin{equation}\label{dist_facce_transversali}
(\de_{F_i}\ast\de_{-F_j}) (\phi)=
\frac1{\sqrt{1-{(\nu_i\cdot \nu_j)}^2}}
\int_{\Real^{n}}\cH^{n-2}(F_i\cap(F_j+z))\phi(z) dz.
\end{equation}
When $n=2$ formula~\eqref{dist_facce_transversali} is proved  as above by adapting the notations. The formulas \eqref{somma_dirac}, \eqref{dist_facce_parallele} and \eqref{dist_facce_transversali} imply \eqref{second_derivatives}.

By~\eqref{second_derivatives}, there exists $C\in\Real$ such that $-(\pa^2 g_P/\pa w^2,\phi)\leq C\sup_{\Real^n}|\phi|$ for each $\phi\in C^\infty_0(\Real^n)$. Therefore, by~\cite[Th.~2.1.6]{H}, the distribution $-\pa^2 g_P/\pa w^2$ has an unique extension to a bounded linear functional on $C_c(\Real^n)$, the space of functions on $\Real^n$ which vanish at infinity endowed with the supremum norm.
The Riesz representation Theorem~\cite[p.~49]{EG} implies the existence of a  Radon measure $\mu$ and a $\mu$-measurable function $\si$, with $|\si|=1$ $\mu$-almost everywhere, such that $-(\pa^2 g_P/\pa w^2,\psi)=\int_{\Real^n}\psi \,\si\,d\mu$ for each $\psi\in C_c(\Real^n)$. By Lebesgue decomposition Theorem~\cite[p.~42]{EG} the measure  $\mu$ has an unique decomposition $\mu=\mu_{ac}+\mu_s$, where $\mu_{ac}$ is absolutely continuous with respect to $\la_n$ and $\mu_s$ and $\la_n$ are mutually singular. The first sum in the right hand side of \eqref{second_derivatives} coincides with $\int\phi\,\si\,d\mu_{ac}$, while the second sum coincides with $\int\phi\,\si\,d\mu_s$. Both sums are thus uniquely determined by $g_P$.
\end{proof}

\begin{proof}[Proof of Proposition~\ref{rufibach}]
Let $F_i$, $\nu_i$ and $I_p$ be as in the statement of Lemma~\ref{deriv_distrib}.  Consider the distribution defined by the second sum in \eqref{second_derivatives}. This distribution determines its support  $S(P,w)=\cup_{(i,j)\in I_{p}:\nu_i\cdot w\neq0}(F_i-F_j)$ and
\begin{equation}\label{parte_singolare_dist}
\sum_{(i,j)\in I_{p}} w\cdot\nu_i\ w\cdot\nu_j\ 
\cH^{n-1}(F_i\cap(F_j+x)), 
\end{equation}
for each $x\in S(P,w)$. Clearly we have $S(P,w)\subset D\,P$.
If $F_i$ and $F_j$ are parallel and $i\neq j$ then $\nu_i=-\nu_j$ and, by \eqref{facce_corpodiff}, $F_i-F_j$ is the facet $(D\,P)_{\nu_i}$ of $D\,P$. Moreover $F_i-F_i\subset\nu_i^\perp$. Thus we have
\begin{equation}\label{supp_sing}
S(P,w)\cap \inte D\,P=\cup_{i:\nu_i\cdot w\neq0}(F_i-F_i)\subset\cup_{i:\nu_i\cdot w\neq0} \nu_i^\perp
\end{equation}
and $\cH^{n-1}(S(P,w)\cap \nu_i^\perp)>0$ when $\nu_i\cdot w\neq0$.

If $\cH^{n-1}(S(P,w)\cap w^\perp)=0$ then neither $F$ nor $G$ are facets of $P$ and $g_\tF+g_\tG\equiv0$. Now assume $\cH^{n-1}(S(P,w)\cap w^\perp)>0$. In this case $w$ coincides, up to the sign, with one of the $\nu_i$. 
If $x\in (w^\perp\setminus\cup_{\nu_j\neq\pm w} \nu_j^\perp)\cap \inte D\,P$ then, by \eqref{supp_sing}, the expression in \eqref{parte_singolare_dist} coincides with  $\cH^{n-1}(F\cap(F+x))+\cH^{n-1}(G\cap(G+x))=\cH^{n-1}(\tF\cap(\tF+x))+\cH^{n-1}(\tG\cap(\tG+x))=g_\tF(x)+g_\tG(x)$. On the other hand, if $x\in w^\perp\setminus D\:P$ we have $g_\tF(x)+g_\tG(x)=0$, because $P\cap(P+x)=\emptyset$ implies $F\cap(F+x)=\emptyset$ and
$G\cap(G+x)=\emptyset$. Since $g_\tF+g_\tG$ is continuous, this function is determined for all $x\in w^\perp$ by continuity.

Consider $(D\,P)_w$. If it is not contained in $S(P,w)$ then either $F$ or $G$ is not a facet of $P$ and $g_{\tF,\tG}\equiv0$. If it is contained then $F=F_i$, $G=F_j$, for some $i$ and $j$ with $i\neq j$, and $(D\,P)_w=F_i-F_j$. In this case the expression in \eqref{parte_singolare_dist} coincides with  $\cH^{n-1}(F_i\cap(F_j+x))$.    Knowing this function for each $x\in (D\,P)_w$ is equivalent to knowing $g_{\tF,\tG}$.
\end{proof}

\begin{lemma}\label{lemma_sistema_cov}
Let $F$, $F'$, $G$ and $G'$ be convex bodies in $\Real^n$ with $\inte F\neq\emptyset$.
If $g_F=\al g_{F'}$, for some $\al\neq0$,  then $\al=1$.
If 
\begin{equation}\label{system_cov}
\begin{cases}
&g_F+g_G=g_{F'}+g_{G'},\\
&g_{F,G}=g_{F',G'}
\end{cases}
\end{equation}
then  either  $g_{F}=g_{F'}$ and $g_{G}=g_{G'}$, or else $g_{F}=g_{G'}$ and $g_{G}=g_{F'}$.
\end{lemma}
\begin{proof}Observe that  if $K\subset\Real^n$ is a convex body then $\la_n(K)=g_K(0)$ and 
\[
\la_n^2(K)=\int_{\Real^n}\int_{\Real^n}1_K(y)1_K(y-x)\:dy\:dx=\int_{\Real^n}g_K(x)\:dx
\]
(see \eqref{fourier} and~\cite[p.~411]{Sc}). Thus the identity $g_F=\al g_{F'}$ implies $\la_n (F)=\al\la_n (F')$, $\la_n^2(F)=\al\la_n^2(F')$ and, as a consequence, $\al=1$.

Let us prove the second claim. Applying the Fourier transform to the equalities in \eqref{system_cov}  we arrive, with the help of \eqref{fourier}, to the system
\begin{equation*}
\begin{cases}
&\|\widehat{1_{F}}\|^2+\|\widehat{1_{G}}\|^2=
\|\widehat{1_{F'}}\|^2+
\|\widehat{1_{G'}}\|^2\\
&\|\widehat{1_{F}}\|^2
 \|\widehat{1_{G}}\|^2
=\|\widehat{1_{F'}}\|^2
 \|\widehat{1_{G'}}\|^2.
\end{cases}
\end{equation*}
Let $\xi\in\Real^n$ denotes the transform variable. For each $\xi\in\Real^n$, the previous system implies that  either we have
$\|\widehat{1_{F}}(\xi)\|=\|\widehat{1_{F'}}(\xi)\|$
and
$\|\widehat{1_{G}}(\xi)\|=\|\widehat{1_{G'}}(\xi)\|$
or else we have
$\|\widehat{1_{F}}(\xi)\|=\|\widehat{1_{G'}}(\xi)\|$
and
$\|\widehat{1_{G}}(\xi)\|=\|\widehat{1_{F'}}(\xi)\|$. The alternative a priori may depend on $\xi$.
The Fourier transform of a function with compact support is analytic (see~\cite[Th.~7.1.14]{H}) and therefore the squared moduli of the previous transforms are analytic. Since any analytic
function is determined by its values on a set with a limit point, we conclude that the previous alternative does not depend on $\xi$.
Going back to covariograms via Fourier inversion, this means that either $g_{F}=g_{F'}$ and $g_{G}=g_{G'}$, or else $g_{F}= g_{G'}$ and $g_{G}=g_{F'}$.
\end{proof}

\begin{remark}\label{rem_facce} Let $P$, $F$, $G$, $\tF$ and $\tG$ be as in Proposition~\ref{rufibach}. Assume $n=3$ and $F$ and $G$ facets,  let $P'$ be a convex polytope with $g_P=g_{P'}$ and let $F'=P'_w$, $G'=P'_{-w}$, $\tF'=F'\mathbin|w^\perp$ and $\tG'=G'\mathbin|w^\perp$. Proposition~\ref{rufibach}, Lemma~\ref{lemma_sistema_cov} and the positive answer to the covariogram problem in the plane imply that, possibly after a reflection of $P'$, $\tF'$ and $\tG'$ are translations or reflections respectively of $\tF$ and $\tG$. Ruling out the possibility that, say, $\tF=-\tF'\neq\tF'$ and $\tG=-\tG'\neq\tG'$ is a major difficulty in the proof of Theorem~\ref{teorema_lmr}, and to overcome it we need Theorem~\ref{cov_congiunto_poligoni}.
This possibility cannot be overcome when $n\geq 4$; see Remark~\ref{rem_facce_four}.
\end{remark}

\begin{proof}[Proof  of \eqref{t_lmr_facce} in Theorem~\ref{teorema_lmr}]
Let $F=P_w$, $G=P_{-w}$, $F'=P'_w$ and $G'=P'_{-w}$.
The relations \eqref{support} and \eqref{facce_corpodiff} imply
\begin{equation}\label{diff_facce}
F-G=(D\,P)_w=(D\,P')_w=F'-G'.
\end{equation}
Up to a translation of $P$ and $P'$, a reflection of $P'$ and an affine transformation, we may assume $w=(0,0,-1)$,
\begin{align}
&\dime F\geq \dime G,\quad\dime F'\geq \dime G',\\
&F, F'\subset\set{x:x_3=0},\quad G, G'\subset \set{x:x_3=1},\\
&\text{$s(F)=s(F')=O$ and $s(G)=s(G')=(0,0,1)$.}\label{steiner_points}
\end{align}
Here $x=(x_1,x_2,x_3)\in\Real^3$  and we have used \eqref{diff_facce} and the Minkowski-additivity of the Steiner point  (see~\cite[p.~42]{Sc}) to obtain  \eqref{steiner_points}.

Let $\ee>0$ and let $\tee=(0,0,-1+\ee)$.
We have $P\cap(P+\tee)\subset\set{x: 0\leq x_3\leq \ee}$. We study the asymptotic behaviour of the volume of this set 
as $\ee$ tends to $0^+$.

Let $\tG=G \mathbin| \{x:x_3=0\}$ and $\tG'=G'\mathbin| \{x:x_3=0\}$. Observe that $O\in\relint F\cap\relint\tG$ and $O\in\relint F'\cap\relint\tG'$, since  $s(F)=s(F')=s(\tG)=s(\tG')=O$ and the Steiner point of a convex body belongs to its relative interior; see~\cite[p.~43]{Sc}. 
According to the dimension of $F$ and $G$, we distinguish the following cases ($c$ denotes a positive constant which may vary from formula to formula).

\textit{Case 1: $F$ and $G$ are facets.}  We have  $g_P(\tee)=\ee\, \area(F\cap\tG)+o(\ee)$. Indeed $P\cap(P+\tee)$ coincides (up to polytopes of volume $o(\ee^2)$) with  the sum of the  polygon $F\cap\tG$ and  the segment $[O,(0,0,\ee)]$.

\textit{Case 2: $F$ is a facet and $G$ is an edge.} We have $g_P(\tee)=c\,\ee^2 \,\len(F\cap\tG)+o(\ee^2)$. Indeed $P\cap(P+\tee)$ coincides (up to polytopes of volume $o(\ee^3)$) with the sum of the segment $F\cap\tG$ and a triangle with edge lengths proportional to $\ee$ contained in a plane orthogonal to $F\cap\tG$.

\textit{Case 3: $F$ is a facet and $G$ is a vertex.} We have  $g_P(\tee)=c\,\ee^3$, since $P\cap(P+\tee)$ is a pyramid with edge lengths proportional to $\ee$.

\textit{Case 4: $F$ and $G$ are parallel edges.} We have $g_P(\tee)=c\,\ee^2  +o(\ee^2)$, because $P\cap(P+\tee)$ coincides  (up to polytopes of volume $o(\ee^3)$) with the sum of the segment $F\cap\tG$ and a quadrilateral with edge lengths  proportional  to $\ee$, contained in a plane orthogonal to $F\cap\tG$.

\textit{Case 5: $F$ and $G$ are non-parallel edges.} We have $g_P(\tee)=c\,\ee^3$, because $P\cap(P+\tee)$ is a tetrahedron with edge lengths proportional to $\ee$.

\textit{Case 6: $F$ is an edge and $G$ is a vertex.} We have $g_P(\tee)=c\,\ee^3$, because $P\cap(P+\tee)$ is  a polytope with edge lengths proportional to $\ee$.

\textit{Case 7: $F$ and $G$ are vertices}. This is the only case where $(D\,P)_w$ is a point.

In Case~3 we have $g_F+g_\tG=g_F\not\equiv0$ while, in Case~5, we have $g_F+g_\tG\equiv0$.

\bigskip
\begin{tabular}{cllll}
\toprule case &$F$, $G$ & main term of $g_P(\tee)$ & $(D\,P)_w$ & $g_F+g_\tG$\\
\midrule 
 1& facet, facet & $\area(F\cap\tG)\ \ee$ &  facet &\\
 2& facet, edge &$c\, \len(F\cap\tG)\, \ee^2$ & facet & \\
 3& facet, vertex &$c\, \ee^3$ & facet & $\not\equiv0$ \\
 4& parallel edges   & $c\, \ee^2$ & edge & \\
 5& non-parallel edges  & $c\, \ee^3$ & facet & $\equiv0$ \\
 6& edge, vertex   & $c\, \ee^3$ & edge & \\
 7& vertex, vertex   & & vertex  &\\
\bottomrule
\end{tabular}
\bigskip

The information summarised in the last three columns of this table is provided by $g_P$ (recall Proposition~\ref{rufibach} and \eqref{support}). This information distinguishes each case from the others. To conclude the proof it suffices to show that in each case we have $F=F'$ and $G=G'$, possibly after a reflection of $P'$ about $(0,0,1/2)$. 

Case 1. 
Proposition~\ref{rufibach} implies $g_{F,\tG}=g_{F',\tG'}$.
If $(F,\tG)$ and $(F',\tG')$ are trivial associates, then, possibly after a reflection of $P'$ about $(0,0,1/2)$, we have $F=F'+y$ and $\tG=\tG'+y$, for some $y\in\{x:x_3=0\}$. The assumption  \eqref{steiner_points} implies  $y=0$, because $s(F'+y)=s(F')+y$; see~\cite[p.~43]{Sc}. 

Now assume that $(F,\tG)$ and $(F',\tG')$ are not trivial associates. Theorem~\ref{cov_congiunto_poligoni} states that $(F,\tG)$ and $(F',\tG')$ are respectively trivial associates of $(\cT \cK_i, \cT \cL_i)$ and $(\cT \cK_j, \cT \cL_j)$, for some affine transformation $\cT$ and different indices $i,j$, with either $i,j\in\set{1,2}$ or $i,j\in\set{3,4}$. Proposition~\ref{rufibach} implies  $g_{\cT\cK_i}+g_{\cT\cL_i}=g_{\cT\cK_j}+g_{\cT\cL_j}$. Lemma~\ref{lemma_sistema_cov} and the positive answer to the covariogram problem in the plane~\cite{AB2} imply that either $\cK_i$ is a translation or a reflection of $\cK_j$ and $\cL_i$ is a translation or a reflection of $\cL_j$, or  else $\cK_i$ is a translation or a reflection of $\cL_j$ and $\cL_i$ is a translation or a reflection of $\cK_j$. This is clearly false.

Case 2. 
We have $g_F=g_F+g_\tG=g_{F'}+g_{\tG'}=g_{F'}$, which implies either $F=F'$ or $F=-F'$ (recall, again, \eqref{steiner_points}). When $F=F'$, \eqref{diff_facce} implies $G=G'$, since the Minkowski addition satisfies a cancellation law. When $F=-F'$, \eqref{diff_facce} implies
\begin{equation}\label{addendi}
F'+\tG=-F'+\tG'.
\end{equation}
We claim that $\tG=\tG'$ (and $F'=-F'$). Observe that \eqref{steiner_points} implies that $O$ is the midpoint of $\tG$ and of $\tG'$. Identity~\eqref{addendi} implies, for each $u\in S^1$,
\[
h_{F'}(u)-h_{-F'}(u)=h_{\tG'}(u)-h_{\tG}(u).
\]
The function in the right hand side is even, since $\tG$ and $\tG'$ are $o$-symmetric. The function in the left hand side is odd, since $h_{-F'}(u)=h_{F'}(-u)$. Thus both functions vanish and $\tG=\tG'$.
Again, \eqref{addendi} implies $F=F'$.

Cases 3, 6 and 7. 
In each of these cases $G$ and $G'$ are vertices and, by \eqref{steiner_points}, $G=G'=\{(0,0,1)\}$.
Identity \eqref{diff_facce} implies $F=F'$ too.

Case 4. 
The face $(D\,P)_w$ determines the direction of the edges $F$ and $G$ and the sum of their lengths, because $\len((D\,P)_w)=\len(F)+\len(G)$. Thus $F$, $G$, $F'$ and $G'$ are parallel and $\len(F)+\len(G)=\len(F')+\len(G')$.
On the other hand, if $q\in\{x: x_3=0\}$ is parallel to $(D\,P)_w$ we have  $g_P(\tee+q)=c\,\ee^2 \len(F\cap(\tG+q))+o(\ee^2)$, where the strictly positive constant $c$ does not depend on $q$ (it depends only on the ``openings'' of the dihedral cones $\cone(P,F)$ and $\cone(P,G)$). Thus we have
\[
c\,\len(F\cap(\tG+q))=c'\,\len(F'\cap(\tG'+q)),
\]
where the constant $c'$ may a priori differ from $c$.
The term $c\,\len(F\cap(\tG+q))$ coincides with $c\,\len(F\cap\tG)$  when $q$ satisfies $2\|q\|\leq \al: =\max(\len(F),\len(G))-\min(\len(F),\len(G))$, and it is strictly less than $c\,\len(F\cap\tG)$ when $2\|q\|>\al$. Similar considerations hold also for $c'\,\len(F'\cap\tG')$. Thus we have 
\[
\al=\max(\len(F'),\len(G'))-\min(\len(F'),\len(G')).
\]
Thus we have $F=F'$ and $G=G'$,
up to a reflection of $P'$ about $(0,0,1/2)$.

Case 5. The face $(D\,P)_w$ is a parallelogram and therefore has an unique decomposition as  Minkowski sum of two summands, except for the order of the summands. Therefore \eqref{diff_facce} implies  $F=F'$ and $G=G'$, up to a reflection of $P'$ about $(0,0,1/2)$.
\end{proof}


\section{The cross covariogram problem for cones}
Let $A$, $A'$, $B$ and $B'$ be convex polyhedral cones in $\Real^3$ with non-empty interior and such that
\begin{equation}\label{contesto_ab}
\begin{aligned}
&A, A', -B, -B'\subset \set{x\;:\;x_3\geq0},\\ 
&A\cap\set{x : x_3=0}=B\cap\set{x : x_3=0}=\{O\},\\
&A'\cap\set{x : x_3=0}=B'\cap\set{x : x_3=0}=\set{O},
\end{aligned}
\end{equation}
where $x=(x_1,x_2,x_3)\in\Real^3$.
Consider the polygons
\begin{equation}\label{def_F_G}
\begin{aligned}
&F=A\cap \set{x : x_3=1}\,,& &F'=A'\cap \set{x : x_3=1}\,,\\
&G=(-B)\cap \set{x : x_3=1}\,,& &G'=(-B')\cap \set{x : x_3=1}\:.\end{aligned}
\end{equation}
It is easy to see that in this setting $A-B=\conv(A\cup(-B))$ and $A'-B'=\conv(A'\cup(-B'))$. Therefore, $\gab=\gabp$ implies, by \eqref{support},
\begin{equation}\label{identity_conv}
 H:=\conv(F\cup G)=\conv(F'\cup G').
\end{equation}

\begin{proposition}\label{proposizione_coni}Let $A$, $A'$, $B$, $B'$, $F$, $F'$, $G$, $G'$ and $H$ be as above, and assume $g_{A,B}=g_{A',B'}$. Let $z_1,\dots,z_n$ denote the vertices of $H$ in counterclockwise order.  Assume that the following assumptions hold, for each $i=1,\dots,n$:
\begin{enumerate}
\item\label{pc_1} the point $z_i\in F\cap G$ if and only if $z_i\in F'\cap G'$;
\item\label{pc_2}  the segment $[z_i,z_{i+1}]$ is an edge of $F$ or of $G$ if and only if it is an edge of $F'$ or of $G'$;
\item\label{pc_3} if $z_i\notin F\cap G$, then the polygon, between $F$ and $G$, which contains $z_i$ coincides in a neighbourhood of $z_i$ with the polygon, between $F'$ and $G'$, which contains $z_i$. 
\end{enumerate}
Then either $A=A'$ and $B=B'$ or else $A=-B'$ and $B=-A'$.
\end{proposition}
To prove this result we need some preliminary lemmas.

\subsection{Covariogram, $X$-rays of cones and $-1$-chord functions of their sections}
The next lemma is the crucial result that connects covariogram and $X$-rays.

\begin{lemma}\label{derivate_cono_diedro}
Let $L\subset\Real^3$ be a dihedral cone, $R_0$ be its edge and $R_1$ and $R_2$ be its facets. For $i=1,2$, let $v_i\in S^2\cap \relint R_i$. If $\tilde A\subset\Real^3$ is a convex cone,  $\tilde A \cap L=\set{O}$, and $t$, $s\in\Real$ are chosen so that $A\cap(L+tv_1+sv_2)=\emptyset$ or $\inte A\cap(L+tv_1+sv_2)\neq\emptyset$, then
\begin{equation}\label{derivate_miste}
\frac{\pa^2}{\pa s\pa t} \volt(\tilde A\cap(L+tv_1+sv_2))=\al\  \len(\tilde A\cap (R_0+tv_1+sv_2)).
\end{equation}
Here $\al$ is a positive constant which does not depend  on $\tilde A$.
\end{lemma}
\begin{proof}
Assume $L=\set{x : x_1\geq0, x_2\geq0}$, $v_1=(1,0,0)$ and $v_2=(0,1,0)$. Standard calculus arguments prove the formulas
\begin{align*}
\frac{\pa}{\pa t}\volt(\tilde A\cap\set{x: x_1\geq t,x_2\geq s})
&=-\area(\tilde A\cap\set{x: x_1=t, x_2\geq s}),\\
\frac{\pa}{\pa s}\area(\tilde A\cap\set{x: x_1=t, x_2\geq s})&= -\len(\tilde A\cap\set{x: x_1=t, x_2=s}),
\end{align*}
whenever $A\cap\set{x: x_1=t, x_2= s}=\emptyset$ or $\inte A\cap\set{x : x_1=t, x_2= s}\neq\emptyset$.
These identities imply \eqref{derivate_miste} with $\al=1$. In the general case the result follows from a reduction to the previous  one via a non-degenerate linear transformation $\calA$ such that $\calA(L)=\set{x : x_1\geq0, x_2\geq0}$, $\calA(v_1)=(1,0,0)$ and $\calA(v_2)=(0,1,0)$. Indeed we have
\begin{align*}
\frac{\pa^2}{\pa s\pa t} \volt(\tilde A\cap(L+tv_1+sv_2))&= |\det \calA|\,\frac{\pa^2}{\pa s\pa t} \volt(\calA^{-1} (\tilde A)\cap\set{x: x_1\geq t, x_2\geq s})\\
&=|\det \calA|\, \len(\calA^{-1} (\tilde A)\cap\set{x: x_1=t, x_2=s})\\
&=|\det \calA|\,\|\calA(0,0,1)\|^{-1}\,\len(\tilde A\cap(R_0+tv_1+sv_2)).
\end{align*}
\end{proof}

\begin{lemma}\label{equal_chord_functions}
Assume there exists $i\in\{1,\dots,n\}$ such that $z_i\notin F\cap G$. Then the polygon, between $F$ and $G$, which does not contain $z_i$, and the polygon, between $F'$ and $G'$, which does not contain $z_i$ have equal $-1$-chord functions at $z_i$. 
\end{lemma}

\begin{proof} We prove, for instance, that if $z_i\in(G\setminus F)\cap (G'\setminus F')$ then $F$ and  $F'$ have the same $-1$-chord functions at $z_i$.
Let $R$ be the edge of  $B$ and $B'$ with the property that $z_i$ is collinear to $R$, and let $l$ be the line containing $z_i$ and $R$. The dihedral cones $L:=\cone(B,R)$ and $\cone(B',R)$ coincide, due to hypothesis \eqref{pc_3} in Proposition~\ref{proposizione_coni}.  We claim that 
\begin{equation}\label{AintersectL}
A\cap(B+x)=A\cap(L+x)\quad\text{and}\quad A'\cap(B'+x)=A'\cap(L+x),
\end{equation}
for each $x$ in a suitable neighbourhood $V$ of  $z_i$. 
Let $\pi$ be any plane through $O$ which strictly supports $H$ at $z_i$, and let $\pi^+$ be the closed halfspace bounded by $\pi$ not containing $H$. We have $A\cap\pi^+=\{O\}$ and $B\subset\pi^+$. Since $l\subset\pi$, we also have $B+z_i\subset\pi^+$ and $O\in R+z_i$. These arguments imply $A\cap(B+z_i)=\{O\}$. Therefore, when $x$ is close to $z_i$, we have $A\cap(B+x)=A\cap(L+x)$. Similar arguments prove  $A'\cap(B'+x)=A'\cap(L+x)$.
The identities \eqref{AintersectL} imply
\begin{equation*}
\volt(A\cap(L+x))=\gab(x)=\gabp(x)=\volt(A'\cap(L+x))\,,
\end{equation*}
for each $x\in V$.
The latter and Lemma~\ref{derivate_cono_diedro} imply  
\[
 \la_1(A\cap (l+y))=\la_1(A'\cap (l+y)),
\]
for all $y$ in a neighbourhood of $O$ such that $l+y$ meets $\inte A$ or does not meet $A$, and, moreover,  $l+y$ meets $\inte A'$ or does not meet $A'$. Since the left and the right hand side  in the previous formula are homogeneous functions of $y$ of degree $1$, and they are concave on their supports, the previous identity holds for all $y$, that is, $A$ and $A'$ have equal $X-$rays in the direction of $l$.
The passage from $X$-rays to  $-1$-chord functions comes from~\cite[Th.~1.3]{B3}, which proves that if two cones $A$ and $A'$ have the same $X$-rays in the direction of $l$  then  their sections $F$ and $ F'$ with the plane $\set{x:x_3=1}$  have the same $-1$-chord functions at $l\cap\set{x:x_3=1}$, that is, at  $z_i$.
\end{proof}

\subsection{The set of $C^3$ discontinuities of the covariogram}
\begin{lemma}\label{discontinuita_diedri}
Let $C\subset \Real^3$ be a dihedral cone with edge the $x_1$ axis, let  $D\subset \Real^3$ be a dihedral cone with edge the $x_2$ axis and assume that no facet of $C$ or of $D$ is contained in $\{x : x_3=0\}$. For $t\in\Real$, let
\[
g(t)=\volt\Big(A\cap(D+(0,0,t))\cap B(0,1)\Big).
\]
Then $d^3g/dt^3$ is discontinuous at $t=0$. More precisely, if both $C$ and $D$ meet both $\{x : x_3>0\}$ and $\{x : x_3<0\}$ then
\[
\lim_{t\to0^+}\frac{d^3g}{dt^3}(t)>\lim_{t\to0^-}\frac{d^3g}{dt^3}(t),
\]
while if $C$ meets both $\{x : x_3>0\}$ and $\{x : x_3<0\}$ and $D\subset\{x : x_3\geq 0\}$  then 
\[
\lim_{t\to0^+}\frac{d^3g}{dt^3}(t)<\lim_{t\to0^-}\frac{d^3g}{dt^3}(t).
\]
\end{lemma}
\begin{proof}
Assume $C\subset\{x : x_3\leq 0\}$ and $D\subset \{x : x_3\geq 0\}$. In this case $C\cap (D+(0,0,t))$ is empty when $t>0$, and it is a tetrahedron of edge lengths proportional to $|t|$ when $t<0$ and $|t|$ is small. Thus
\[
g(t)=-\al t^3 1_{(-\infty,0]},
\]
for each $t$ in a neighbourhood of $0$ and for some $\al>0$, and we have $\lim_{t\to0^+}d^3g(t)/dt^3>\lim_{t\to0^-}d^3g(t)/dt^3$. 

Now assume that $C$ meets both $\{x : x_3>0\}$ and $\{x : x_3<0\}$, while $D\subset \{x : x_3\geq 0\}$. Let $C'\subset\{x : x_3\leq0\}$  be a closed dihedral cone with edge the $x_1$ axis such that $\inte C\cap \inte C'=\emptyset$ and  $C\cup C'$ is an halfspace $\pi^+$. Clearly 
\[
g(t)=\volt(\pi^+\cap(D+(0,0,t))\cap B(0,1))-\volt (C'\cap (D+(0,0,t))\cap B(0,1)).
\]
Since the first term in the right hand side of the formula is a $C^3$ function of $t$ and since the previous case applies to the second term, we have $\lim_{t\to0^+}d^3g(t)/dt^3<\lim_{t\to0^-}d^3g(t)/dt^3$.

Now assume that both $C$ and $D$ meet both $\{x : x_3>0\}$ and $\{x : x_3<0\}$. Let $D'\subset\{x : x_3\geq0\}$  be a closed dihedral cone with edge the $x_2$ axis such that $\inte D\cap \inte D'=\emptyset$ and  $D\cup D'$ is an halfspace. Arguing as above one writes $g(t)$ as a $C^3$ function minus $\volt(C\cap(D'+(0,0,t))\cap B(0,1))$. Since the previous case applies to this last function, we have $\lim_{t\to0^+}d^3g(t)/dt^3>\lim_{t\to0^-}d^3g(t)/dt^3$.
When both $C$ and $D$ are contained in $\{x : x_3\geq 0\}$ similar ideas prove the claim.
\end{proof}

\begin{lemma}\label{discontinuita_c3}
Let $A$ and $B$ be convex polyhedral cones in $\Real^3$ with non-empty interior satisfying \eqref{contesto_ab}, let $S^3(A,B)=\cl\set{x\in\Real^3\;:\;\text{$\gab$ fails to be $C^3$ at $x$}}$ and $E(A,B)=\{R+T\;:\;\text{$R$ is an edge of $A$ and $T$ is an edge of $-B$}\}$. Then
\[
E(A,B)\subset S^3(A,B)\subset \pa A \cup(-\pa B)\cup E(A,B).
\]
\end{lemma}
\begin{proof}
Let $W\subset\Real^3\setminus\big(\pa A \cup(-\pa B)\cup E(A,B)\big)$ be a connected set. When $x\in W$ neither the vertex $x$ of $B+x$ belongs to $\pa A$, nor the vertex $O$ of $A$ belongs to $\pa B+x$, nor an edge of $A$ intersects an edge of $B+x$. Thus the combinatorial structure of $\pa(A\cap(B+x))$ does not change in $W$. Since the vertices of $A\cap(B+x)$ are smooth functions of $x$, for each $x\in W$, so is $g_{A,B}$. This proves the inclusion $S^3(A,B)\subset \pa A \cup(-\pa B)\cup E(A,B)$.

We prove $E(A,B)\subset S^3(A,B)$. The set $E(A,B)$ is contained in a finite set $\mathcal E$ of planes through $O$ whose intersection with $E(A,B)$ has dimension $2$. For $\pi\in\mathcal E$ let
\begin{multline*}
E_0(\pi)=\{R:\text{$R\subset\pi$ is an edge of $A$}\}\cup\{T:\text{$T\subset\pi$ is an edge of $-B$}\}\cup\\
\cup\{(R+T)\cap\pi: \text{$R$ and $T$ are edges of $A$ or of $-B$ not contained in $\pi$}\}.
\end{multline*}
It suffices to prove  $E(A,B)\setminus\cup_{\pi\in\mathcal E}E_0(\pi)\subset S^3(A,B)$, since $\cl(E(A,B)\setminus\cup_{\pi\in\mathcal E}E_0(\pi))=E(A,B)$ (because $\cup_{\pi\in\mathcal E}E_0(\pi)$ is a finite union of rays) and $S^3(A,B)$ is closed.

Let $\pi\in\mathcal E$, $x_0\in\pi\cap(E(A,B)\setminus E_0(\pi))$ and let $v\in S^2$ be orthogonal to $\pi$. Assume that $\pi$ contains two different edges $R_1$ and $R_2$ of $A$,  two different edges $T_1$ and $T_2$ of $B$, and $\pi$ does not contain any facet of $A$ or of $B$.  The choice  $x_0\in\pi\cap E(A,B)$ implies that at least two of the  sets $R_i\cap (T_j+x_0)$, $i,j=1,2$, are non-empty. Assume, for instance, that all four sets are non-empty and  let $\{p_{i,j}\}=R_i\cap (T_j+x_0)$. For $\ee>0$ sufficiently small, $i,j=1,2$ and $x$ in a neighbourhood  of $x_0$, let
\begin{align*}
g_0(x)&=\volt\Big(A\cap(B+x)\cap\big(\Real^3\setminus\cup_{i,j=1}^2 B(p_{i,j},\ee)\big)\Big),\\
g_{i,j}(x)&=\volt\big(A\cap(B+x)\cap B(p_{i,j},\ee)\big).
\end{align*}
Clearly $g_{A,B}=g_0+\sum_{i,j=1}^2g_{i,j}$. The choice $x_0\notin E_0(\pi)$ implies that  $x_0\notin\pa A$, $O\notin(\pa B+x_0)$ and that the points $p_{i,j}$ are the only intersections of edges of $A$ with edges of $B+x_0$. Thus, arguments similar to those used in the first part of the proof imply $g_0$  $C^3$ in a neighbourhood of $x_0$. Moreover, Lemma~\ref{discontinuita_diedri} proves that $\sum_{i,j=1}^2 g_{i,j}$ fails to be $C^3$ at $x_0$. 
Similar arguments prove that $g_{A,B}$ fails to be $C^3$ at $x_0$ when $\pi$ does not contain any facet of $A$ or of $B$ and $\pi$ contains one or two edges  of $A$ and one or two edges of $B$.

Now assume that $\pi$ contains a facet $R$ of $A$, two different edges $T_1$ and $T_2$ of $B$, and $\pi$ does not contain any facet of $B$. The choice  $x_0\in\pi\cap E(A,B)$ implies that at least one between  $\len(R\cap (T_1+x_0))$ and $\len(R\cap (T_2+x_0))$ is positive. Assume, for instance, that both terms are positive. For $\ee>0$ sufficiently small, $j=1,2$ and $x$ in a neighbourhood of $x_0$, let  $W_j=\cup_{p\in R\cap (T_j+x_0)}B(p,\ee)\setminus B(x_0,\ee)$,
\begin{align*}
g_0(x)&=\volt\Big(A\cap(B+x)\cap\big(\Real^3\setminus(B(x_0,\ee)\cup W_1\cup W_2)\big)\Big),\\
g_{j}(x)&=\volt\big(A\cap(B+x)\cap W_j\big),\\
g_3(x)&=\volt\big(A\cap(B+x)\cap B(x_0,\ee)\big).
\end{align*}
Clearly $g_{A,B}=\sum_{j=0}^3 g_{j}$. The choice $x_0\notin E_0(\pi)$ implies that  $O\notin(\pa B+x_0)$ and that the intersections of edges of $A$ with edges of $B+x_0$ are contained in $W_1\cup W_2$. Thus  $g_0$ is $C^3$ in a neighbourhood of $x_0$. Arguing as in the proof of Lemma~\ref{discontinuita_diedri} proves the following formulas, valid for suitable $\al_1,\al_2,\be>0$:
\begin{gather*}
\left(\lim_{t\to0^+}-\lim_{t\to0^-}\right)\sum_{j=1,2}\frac{d^2g_{j}}{dt^2}(x_0+tv)= \sum_{j=1,2}\al_j\big(\len(R\cap (T_j+x_0))-\ee\big); \\
\left(\lim_{t\to0^+}-\lim_{t\to0^-}\right)\frac{d^2g_{3}}{dt^2}(x_0+tv)\geq-\ee\be.
\end{gather*}
Thus $\sum_{j=1}^3g_{j}$ fails to be $C^2$ at $x_0$ when $\ee>0$ is small enough.
Similar arguments prove that $g_{A,B}$ fails to be $C^2$ at $x_0$ when $\pi$ contains a facet $A$ and one edge (but no facet) of $B$, and also prove that $g_{A,B}$ fails to be $C^1$ at $x_0$ when $\pi$ contains a facet of $A$ and a facet of $B$.
\end{proof}

\begin{remark}
The set $\pa A\cup(-\pa B)$ is not necessarily contained in $S^3(A,B)$, even if generically it is. For instance, let 
\[
A=\pos\conv\{(1,1,1),(1,-1,1),(-1,1,1),(-1,-1,1)\}
\]
and let $B$ be any cone which satisfies \eqref{contesto_ab} and with a face $T$ contained in $\{x: x_2=0\}$.  Then $g_{A,B}$ is $C^3$ at any point $x_0\in\relint T$. This claim relies ultimately on the smoothness of the function $g(t)=\la_3(A\cap\{x:x_2\geq t,x_3\leq1\})=(2/3-t+t^3/3)1_{[-1,1]}$ in $(-1,1)$.
\end{remark}

\subsection{Proof of Proposition~\ref{proposizione_coni}}

We recall a result proved in~\cite{B3}.
\begin{lemma}[\cite{B3}]\label{secondo_lemma_unicita}
Let $K$ and $K'$ be convex polygons  with equal $-1$-chord functions at  $p_1,p_2,\dots,p_s\in \Real^2\setminus K$. If
$\conv(K,p_1,\dots,p_s)=\conv(K',p_1,\dots,p_s)
$
then $K=K'$.
\end{lemma}

\begin{proof}[Proof of Theorem~\ref{proposizione_coni}] 
Within this proof we say that $z_i$ is  \emph{neutral} if $z_i\in F\cap G\cap F' \cap G'$,
that $z_i$ is \emph{concordant} if $z_i\in (F\setminus G)\cap (F'\setminus G')$ or $z_i\in (G\setminus F)\cap (G'\setminus F')$, and that $z_i$ is  \emph{discordant} if $z_i\in (F\setminus G)\cap (G'\setminus F')$ or $z_i\in (G\setminus F)\cap (F'\setminus G')$.
Assumption~\eqref{pc_1} implies that this classification is exhaustive.

\begin{claim}
If one  vertex of $H$ is concordant then no vertex of $H$ is discordant.
\end{claim}

\begin{proof}
Lemma~\ref{discontinuita_c3}, when expressed in terms of $S^3(F,G):=S^3(A,B)\cap\{x : x_3=1\}$, implies
\begin{multline}\label{expression_graph}
\{[a,b]\,:\, \text{$a$ is a vertex of $F$ and $b$ is a vertex of $G$}\}\subset S^3(F,G)\subset\\
\subset\pa F\cup\pa G\cup\{[a,b]\,:\, \text{$a$ is a vertex of $F$ and $b$ is a vertex of $G$}\}.
\end{multline}
Moreover analogous inclusions hold for $S^3(F',G'):=S^3(A',B')\cap\{x : x_3=1\}$.  The identity $g_{A,B}=g_{A',B'}$  implies 
\begin{equation}\label{stre_uguali}
 S^3(F,G)=S^3(F',G').
\end{equation}
Assume the claim false and that $z_i$ is concordant and $z_j$ is discordant, for some $i$ and $j$. Assume also that this means, for instance,
\[
z_i\in (F\setminus G)\cap (F'\setminus G')\quad\text{and}\quad z_j\in (F\setminus G)\cap (G'\setminus F').
\] 
Hypothesis \eqref{pc_2}  implies that  a vertex of $H$ adjacent to a concordant vertex of $H$ is either concordant or neutral. Thus there exist $h$ and $k$ such that $z_h\in(z_i,z_j)_{\pa H}$, $z_k\in(z_j,z_i)_{\pa H}$ and both $z_h$ and $z_k$ are neutral. Without loss of generality, we may also assume  $z_m$ neutral, whenever $z_m\in(z_i,z_j)_{\pa H}$.

We have  $[z_i,z_j]\subset S^3(F',G')$, by \eqref{expression_graph} and because $z_i$ is a vertex of $F'$ and $z_j$ is a vertex of  $G'$. 
Since $z_h\in(z_i,z_j)_{\pa H}$ and $z_k\in(z_j,z_i)_{\pa H}$ are vertices of $F$ and of $G$,  $\relint[z_i,z_j]\cap\pa F=\emptyset$ and $\relint[z_i,z_j]\nsubseteq\pa G$. Therefore \eqref{expression_graph} and \eqref{stre_uguali} imply that $\relint[z_i,z_j]$ must contain a vertex $p$ of $G$. 

Since $z_k\in G$ and $z_m\in G$, whenever $z_m\in(z_i,z_j)_{\pa H}$, then $p$ is not contained in the convex envelope of these points, that is, $p$ is contained in $\inte \conv\{z_k,z_i,z_{i+1}\}$ or in $\inte \conv\{z_k,z_j,z_{j-1}\}$. Assume $p\in\inte \conv\{z_k,z_i,z_{i+1}\}$. 

We have  $[z_k,p]\subset S^3(F,G)$ and, by \eqref{stre_uguali}, $[z_k,p]\subset S^3(F',G')$. Let $q$ be the point of $\relint[z_i,z_{i+1}]$ collinear to $z_k$ and $p$. Since $z_k$ is a vertex of $F'$ and $[z_i,z_{i+1}]$ is an edge of $F'$, $\relint[z_k,q]\cap\pa F'=\emptyset$ and $z_k$ is the only vertex of $F'$ contained in $[z_k,q]$. Thus $[z_k,p]\subset S^3(F',G')$ and \eqref{expression_graph} implies
\[
[z_k,p]\subset\pa G'\cup\{[z_k,b] : \text{$b$ is a vertex of $G'$}\}.
\]
This is possible only if $[p,q]$ contains a vertex $p'$ of $G'$. Since $[z_j,p]\subset\conv\{z_j, z_k, p'\}$ and $z_j, z_k, p'\in G'$, we have $[z_j,p]\subset G'\cap[z_i,z_j]$. On the other hand, since
$p$, $z_h$ and $z_k$ are vertices of $G$ and $z_j\notin G$, we have $G\cap[z_i,z_j]\subsetneq [z_j,p]$. Therefore
\[
G\cap[z_i,z_j]\subsetneq F'\cap[z_i,z_j],
\]
and the $-1$-chord functions of $G$ and $G'$ at $z_i$ in the direction of $z_j-z_i$ differ. This violates the conclusion of Lemma~\ref{equal_chord_functions}. When $p\in\inte \conv\{z_k,z_j,z_{j-1}\}$ similar arguments give a contradiction, by proving that the $-1$-chord functions of $G$ and $F'$ at $z_j$  differ.
\end{proof}

To conclude the proof it suffices to show that either  $F=F'$ and $G=G'$ or else  $F=G'$ and $G=F'$.
If each $z_i$ is neutral then $F=F'=G=G'=H$. Now assume that no $z_i$ is discordant. 
Let $z_{i_1},\dots,z_{i_s}$, for some $s\geq0$, be the vertices of $H$ which are not vertices of $F$. These points are also the vertices of $H$ which are not vertices of $F'$, because no $z_i$ is discordant.   Since 
\[
H=\conv\{F,z_{i_1},\dots,z_{i_s}\}=\conv\{F',z_{i_1},\dots,z_{i_s}\},
\]
the identity $F=F'$ is a consequence of Lemmas~\ref{equal_chord_functions} and~\ref{secondo_lemma_unicita}. One proves $G=G'$ substituting $F$ with $G$ and  $F'$ with $G'$  in the previous arguments. When no $z_i$ is concordant a similar proof gives $F=G'$ and $G=F'$.
\end{proof}

\section{Determining the support cones of $P$: proof of \eqref{t_lmr_coni} in Theorem~\ref{teorema_lmr}}
The next lemma proves \eqref{t_lmr_coni} in a particular case and it is necessary also for its proof in the general case. The idea behind this proof is the following: when the cones to be determined are  support cones in antipodal parallel edges of $P$ of equal length, the problem is substantially two-dimensional and can be reduced to  the one studied in Lemma~\ref{cov_cong_settori_angolari}. 
\begin{lemma}\label{spigoli_opposti_uguali}
Assume that $S_1:=P_w$ and $S_2:=P_{-w}$ are parallel edges of $P$ of equal length, for some $w\in S^2$. Then  $S'_1:=P'_w$ and $S'_2:=P'_{-w}$ are  edges of $P'$ parallel to $S_1$ whose lengths equal  $\len(S_1)$. Moreover,
either $\cone(P,S_1)=\cone(P',S'_1)$ and $\cone(P,S_2)=\cone(P',S'_2)$ or else $\cone(P,S_1)=-\cone(P',S'_2)$ and $\cone(P,S_2)=-\cone(P',S'_1)$.
\end{lemma}
\begin{proof}Formula \eqref{t_lmr_facce} in Theorem~\ref{teorema_lmr} implies that $S'_1$ or $S_2'$ is an edge of $P'$ parallel to $S_1$ whose length is $\len(S_1)$. Assume, for instance, that $S'_2$ is such an edge. Apply again \eqref{t_lmr_facce} with the roles of $P$ and $P'$ exchanged. Either $S_1$ or $S_2$ is a translate of $S_1'$. Thus, also $S_1'$ is an edge of $P'$ parallel to $S_1$ whose  length is $\len(S_1)$.

Let $u\in S^2$ be the direction of $S_1$, $S_2$, $S_1'$ and $S_2'$, let $D_i=\cone(P,S_i)$, $D'_i=\cone(P',S'_i)$, $C_i=D_i\mathbin| u^\perp$ and $C'_i=D'_i\mathbin| u^\perp$, for $i=1,2$. We recall that $l_u$ denotes the line through $0$ parallel to $u$. Since $D_i$ and $D'_i$ are dihedral cones which coincide with $C_i+l_u$ and $C_i'+l_u$, respectively, the lemma is proved once we show that either $C_1=C'_1$ and $C_2=C'_2$ or else $C_1=-C'_2$ and $C_2=-C'_1$.

Let $y\in\Real ^3$ satisfies $S_1=S_2+y$. 
Since conditions \eqref{support} and  \eqref{facce_corpodiff} imply
$S_1-S_2=S_1'-S_2'$, we have $S_1'=S_2'+y$. If  $W_i$ denotes a sufficiently small  neighbourhood of $S_i$ then we can write
\begin{equation}\label{decomp_P}
P\cap W_i=\Big(\big( (C_i+ S_i)\cup E_i^1 \big)\setminus E_i^2\Big)\cap W_i\:,
\end{equation}
where $E_i^1$ and $E_i^2$ are unions of a finite number of convex cones and each of these cones is  contained in $D_i+S_i$, has apex an endpoint of $S_i$ and intersects the line $\aff(S_i)$ only in its apex.  In order to compute $g_P(y+\ee x)$, for $x\in u^\perp\cap S^2$ and $\ee>0$ small, we write
\begin{equation}\label{vol_locale}
P\cap (P+y+\ee x)=(P\cap W_1)\cap\big( (P\cap W_2)+y+\ee x\big)\:.
\end{equation}
Simple elementary calculations lead to the following formulas, for each $i,j=1,2$:
\begin{align*}
&\volt\big((C_1+ S_1)\cap\left(C_2+ S_2+y+\ee x\right)\big)=
\ee^2 \;\area\big(C_1\cap(C_2+x)\big)\;\len(S_1)\:;\\
&\volt\big((C_1+ S_1)\cap(E_2^i+y+\ee x)\big)\leq O(\ee^3)\:;\\
&\volt\big(E_1^i\cap(C_2+ S_2+y+\ee x)\big)\leq O(\ee^3)\:;\\
&\volt\big(E_1^j\cap(E_2^i+y+\ee x)\big)\leq O(\ee^3)\:.
\end{align*}
These formulas,  \eqref{decomp_P} and \eqref{vol_locale} imply
\[
g_P(y+\ee x)=\ee^2 \len(S_1)\;g_{C_1,C_2}(x)+O(\ee^3).
\]
The corresponding asymptotic expansion for $g_{P'}$ is proved by similar arguments. These expansions, the identity $g_P=g_{P'}$ and the homogeneity of degree $2$ of $g_{C_1,C_2}$ and $g_{C_1',C_2'}$  imply 
\begin{equation*}
g_{C_1,C_2}(x)=g_{C_1',C_2'}(x),
\end{equation*}
for each $x\in u^\perp$. Observe that $C_1\cap(-C_2)$ and $C_1'\cap(-C_2')$ are non-empty (it is an immediate consequence of the convexity of $P$ and $P'$) and apply Lemma~\ref{cov_cong_settori_angolari}, with $A$ substituted by $C_1$, $B$ by $C_2$, $A'$ by $C'_1$ and $B'$ by $C'_2$. This lemma implies $\{C_1,-C_2\}=\{C'_1,-C'_2\}$, as explained in Remark~\ref{vettore_in_comune}.
\end{proof}

Now we are ready to prove \eqref{t_lmr_coni} in the general case.

\begin{proof}[Proof of \eqref{t_lmr_coni} in Theorem~\ref{teorema_lmr}]
We distinguish three cases according to $\dim P_w$.

\textit{Case 1. $P_w$ is a facet}. 
In this case \eqref{t_lmr_coni} is an immediate consequence of \eqref{t_lmr_facce}. Indeed, we have
\begin{equation*}
 \cone(P,P_w)=\cone(P',P_w)=\pos\{w\}.
\end{equation*}

\textit{Case 2. $P_w$ is an edge}. 
By \eqref{t_lmr_facce}, we may assume $S:=P_w=P'_w$.
We prove that if $\cone(P,S)\neq\cone(P',S)$ then, for a suitable $y\in\Real^3$, $-S+y$ is an edge of $P'$ and 
\begin{equation}\label{coni_opp_uguali}
\cone(P,S)=-\cone(P',-S+y).
\end{equation}
This clearly implies $\cone(P,S)=\cone(-P'+y,S)$ and $P_w=(-P'+y)_w$. Thus,  \eqref{t_lmr_facce} and \eqref{t_lmr_coni} hold with $-P'+y$ replacing $P'$.
When $\cone(P,S)\neq\cone(P',S)$ we may assume, without loss of generality, the existence of a facet $R$ of $P$ containing $S$ such that $\aff(R)$ supports $P'$ and intersects $\pa P'$ only in $S$. Let $w_1$ be the unit outer normal to $P$ at $R$. We have $P_{w_1}=R$ and $P'_{w_1}= S$ and, therefore, $P'_{w_1}$ is not a translate of $P_{w_1}$. Formula~\eqref{t_lmr_facce} in Theorem~\ref{teorema_lmr},  with $w$ substituted by $w_1$, implies that $(-P')_{w_1}+y=R$, for some $y\in\Real^3$, that is, 
$-R+y$ is a facet of $P'$ with outer normal $-w_1$. In particular, since $S$ is an edge of $R$,  $-S+y$ is an edge of $P'$. Since $-w_1\in N(P',-S+y)$ and $w_1\in \inte N(P',S)$, $S$ and $-S+y$ are  antipodal parallel edges of $P'$ of equal length. Lemma~\ref{spigoli_opposti_uguali} implies \eqref{coni_opp_uguali}, 
since $\cone(P,S)\neq\cone(P',S)$.

\textit{Case 3: $P_w$ is a vertex}.
Up to a small perturbation of $w$, we may assume that also $P_{-w}$ is a vertex. Note that the validity of \eqref{t_lmr_coni} for the perturbed $w$  implies its validity for the original $w$ too. Conditions \eqref{facce_corpodiff} and \eqref{support} imply $P_w-P_{-w}=P'_{w}-P'_{-w}$. Therefore, $P'_{w}-P'_{-w}$ consists of one element, and this may happen only if both $P'_{w}$ and $P'_{-w}$ are points. 

Up to a translation of $P'$ and an affine transformation, we may assume $P_w=P'_w=(0,0,-1)$, $P_{-w}=P'_{-w}=(0,0,1)$,
\begin{equation}\label{coni_contesto}
\begin{aligned}
&P\cup P'\subset \set{x : -1\leq x_3\leq1},\\
&P\cap\set{x : x_3=-1}=P'\cap\set{x : x_3=-1}=\set{(0,0,-1)},\\
&P\cap\set{x : x_3=1}=P'\cap\set{x : x_3=1}=\set{(0,0,1)}.
\end{aligned}\end{equation}

Let $A=\cone(P,(0,0,-1))$, $B=\cone(P,(0,0,1))$, $A'=\cone(P',(0,0,-1))$ and $B'=\cone(P',(0,0,1))$. These cones satisfy \eqref{contesto_ab}. Since we have
\[
g_{A,B}(x)=g_P(x-(0,0,2))=g_{P'}(x-(0,0,2))=g_{A',B'}(x)\,,
\]
for each $x$ in a neighbourhood of $O$, and since $\gab(x)$ and $\gabp(x)$ are  homogeneous functions of degree $3$, we have 
\[
g_{A,B}=g_{A',B'}.
\]

Let $F$, $G$, $F'$, $G'$ and $H$ be defined as in \eqref{def_F_G} and \eqref{identity_conv}. We prove that the part of Theorem~\ref{teorema_lmr} expressed by formula \eqref{t_lmr_facce}  implies that these sets satisfy the assumptions of Proposition~\ref{proposizione_coni}. Once this is done, 
Proposition~\ref{proposizione_coni} implies that \eqref{t_lmr_coni} holds, possibly after a substitution of $P'$ with $-P'$.

\emph{Hypothesis \eqref{pc_3} is satisfied.} It suffices to show this when
\begin{equation}\label{appartenenza_z}
z_i\in(F\setminus G)\cap(F'\setminus G'),
\end{equation}
since in the other cases the proof is similar.
Condition \eqref{appartenenza_z}, when rephrased in terms of the respective cones, states that $\pos(z_i)$ is an edge of $A$, of $A'$ (and of $D:=\conv(A\cup(-B))$) which meets $-B$ and $-B'$ only at $O$. Let $w_1\in S^2\cap \inte N(D,\pos(z_i))$. When rephrased in terms of $P$ and $P'$, \eqref{appartenenza_z} implies that
\begin{equation}\label{nc1}
 (P')_{-w_1}\text{ is a point}
\end{equation}
(it coincides with $(0,0,1)$) and that $S:=P_{w_1}$ and $S':=P'_{w_1}$
are edges of $P$ and $P'$, respectively, with  endpoint $(0,0,-1)$. 

Let us prove $S=S'$. Assume the contrary.
Formula~\eqref{t_lmr_facce} in Theorem~\ref{teorema_lmr}, with $w$ substituted by $w_1$, implies that $S$ is a translate of $(-P')_{w_1}$ (because $S'$ is not a translate of $S$). This contradicts \eqref{nc1} and proves $S=S'$.

The coincidence of $F$ and $F'$ in a neighbourhood of $z_i$ is clearly equivalent to the identity 
\begin{equation}\label{coni_uguali_in-S}
\cone(P,S)=\cone(P',S).
\end{equation} 
Assume that this identity is false. Arguing as in the proof of Case~1, one shows that $-S+y$ is an edge of $P'$ and \eqref{coni_opp_uguali} holds, for a suitable $y\in\Real^3$.
The identity \eqref{coni_opp_uguali} imply $(P')_{-w_1}= -S'+y$ (because $w_1\in\cone(P,S)$), contradicts \eqref{nc1} and proves \eqref{coni_uguali_in-S}. 

\emph{Hypothesis \eqref{pc_1} is satisfied.} Assume $z_i\in F\cap G$. In this case  $\pos(z_i)$ is an edge of $A$, of $-B$ and of $D$. If $w_1\in S^2\cap\inte N(D,\pos(z_i))$, the latter implies that $P_{w_1}$ and $P_{-w_1}$ are antipodal parallel edges of $P$ with endpoints $(0,0,-1)$ and $(0,0,1)$, respectively.  Arguing as in the first lines of the proof of Lemma~\ref{spigoli_opposti_uguali}, one shows that also $P'_{w_1}$ and $P'_{-w_1}$ are edges parallel to $P_{w_1}$. 

Let $\pi$ denote the plane orthogonal to $w_1$ containing $\pos(z_i)$. Since $\pi$ supports $D=\conv(A\cup B)=\conv(A'\cup B')$ at $O$,  $\pi+(0,0,-1)$ supports both $P$ and $P'$ at $(0,0,-1)$, and  $\pi+(0,0,1)$ supports both $P$ and $P'$ at $(0,0,1)$. In particular, we have  $w_1\in N(P',(0,0,-1))$ and $-w_1\in N(P',(0,0,1))$. The latter and  \eqref{coni_contesto} imply that $(0,0,-1)$ and $(0,0,1)$ are endpoints of $P'_{w_1}$ and $P'_{-w_1}$, respectively.
When rephrased in terms of the support cones, this implies that $\pos(z_i)$ is an edge of $A'$ and of $-B'$ or, equivalently, that $z_i$ is a vertex of $F'$ and $G'$.

\emph{Hypothesis \eqref{pc_2} is satisfied.} Assume that $[z_i,z_{i+1}]$ is an edge of $F$. In this case $\pos(z_i)+\pos(z_{i+1})$ is a facet of $A$ and of $D$. If $w_1\in S^2\cap N(D,\pos(z_i)+\pos(z_{i+1}))$, then $P_{w_1}$ is a facet of $P$ with vertex $(0,0,-1)$ and $\pos(P_{w_1}+(0,0,1))=\pos(z_i)+\pos(z_{i+1})$. Arguing as we have done to prove that hypothesis~\eqref{pc_1} is satisfied shows that $(0,0,-1)$ is a vertex of $P'_{w_1}$ and  $(0,0,1)$ is a vertex of $P'_{-w_1}$. This and formula~\eqref{t_lmr_facce} in Theorem~\ref{teorema_lmr} imply that either $P_{w_1}=P'_{w_1}$ or  $-P_{w_1}=(-P')_{w_1}$. In the first case $[z_i,z_{i+1}]$ is an edge of $F'$, in the second case it is an edge of $G'$.
\end{proof}


\section{The structure of synisothetic polytopes}\label{sec_syn}

If $P$ is a translation or a reflection of  $P'$  then $(P,-P)$ is synisothetic to $(P',-P')$, but the converse implication is false.
\begin{example}\label{es_sinisotetici} 
Let $P\subset \Real^3$ be a convex polytope such that $\pa P$  contains a simple closed curve $\Gamma$ together with $-\Gamma$, with $\Gamma\cap(-\Gamma)=\emptyset$. The union $\Gamma\cup(-\Gamma)$ disconnects $\pa P$ in three  components. Let $\Sigma_1$ be the one bounded by $\Gamma$, $\Sigma_2$ the one bounded by $-\Gamma$ and $\Sigma_3$ the one bounded by $\Gamma\cup-\Gamma$. 
Choose $P$ in such a way that $\Sigma_1\neq -\Sigma_2$, $\Sigma_3\neq -\Sigma_3$,  $\pa P$ and $-\pa P$ coincide in a neighbourhood $W$ of $\Gamma$ and $W$ contains all faces which intersect $\Gamma$. 
Let $P'$ be the polytope whose boundary is $(-\Sigma_1)\cup \Sigma_3\cup(-\Sigma_2)$. The polytope $P'$ is not a translate or a reflection of $P$ but $(P,-P)$ is synisothetic to $(P',-P')$. 
\end{example} 

In order to introduce some notations, we need the following two lemmas.

\begin{lemma}\label{corpodiff}
Let $P$ and $P'$ be convex polytopes in $\Real^3$ with non-empty interior such that $(P,-P)$ and $(P',-P')$ are synisothetic. Then $DP=DP'$.
\end{lemma}
\begin{proof}
The  second identity of \eqref{facce_corpodiff} implies
\begin{equation}\label{setting}
(DP)_w=P_w+(-P)_w\quad\text{and}\quad (DP')_w=P'_w+(-P')_w.
\end{equation}
The synisothesis of the two pairs implies that each summand in the right hand side of the first identity also appears in the right hand side of the second identity and vice versa. 
As a consequence we have $\area((DP)_w)=\area((DP')_w)$ for each $w\in S^2$. Therefore $DP$ and $DP'$ have the same $2$-area measure and, by the uniqueness assertion for the  Minkowski Problem~\cite[Th.~7.2.1]{Sc}, they are translates of each other. Since both difference bodies are origin symmetric, they coincide.
\end{proof}

In this section $P$ and $P'$ will always be as in the statement of Lemma~\ref{corpodiff}.
\begin{lemma}\label{lemma_setting}
For each $w\in S^2$  there exist $\si\in\{-1,1\}$ and $x=x(\si)\in \Real^3$ such that
\begin{align}
&\text{$P_w=(\si P')_w+x$ and $\cone(P,P_w)=\cone(\si P',(\si P')_w)$; }\label{alt1}\\
&\text{$P_{-w}=(\si P')_{-w}+x$ and $\cone(P,P_{-w})=\cone(\si P',(\si P')_{-w})$.}\label{alt2}
\end{align}
\end{lemma}
\begin{proof}
Condition~\eqref{alt1} explicitly expresses  the isothesis of $P_w$ and $(\si P)_w$, for some $\si\in\{-1,1\}$, and this holds by assumption. What we have to prove is that $P_{-w}$ is isothetic to $(\si P')_{-w}$ (with the same $\si$) and that the translation that carries $(\si P')_w$ into $P_w$ also carries $(\si P')_{-w}$ into $P_{-w}$.

We know that   $(\si P')_{-w}$ is isothetic either to $P_{-w}$ or to  $(-P)_{-w}$. In the first case the proof regarding $\si$ is concluded. Assume that the second possibility holds. This property is equivalent to the isothesis of $P_{w}$ and  $(-\si P')_{w}$, since for each convex polytope $Q$ we have 
\[
(-Q)_{-w}=-Q_w\quad\text{and}\quad
\cone(-Q,(-Q)_{-w})=-\cone(Q,Q_w).
\]
This and \eqref{alt1} imply  that $P_w$ is isothetic both to $(\si P')_{w}$ and to $(-\si P')_w$. Since isothesis is a transitive property, $(\si P')_{w}$ and  $(-\si P')_w$ are isothetic and in \eqref{alt2} we can choose both $\si=1$ and $\si=-1$.

Lemma~\ref{corpodiff} and \eqref{setting} imply $P_w-P_{-w}=P'_w-P'_{-w}$,
and this identity implies that the translation vector $x$ in \eqref{alt2} coincide with the one in \eqref{alt1}.
\end{proof}
\begin{definition}\label{def_positive}Let $F$ be a proper face of $P$ and $w\in S^2$ be such that $F=P_w$. We say that
\emph{$F$ is positive}  when  \eqref{alt1} holds only with $\si=1$;
that \emph{$F$ is negative} when  \eqref{alt1} holds only with $\si=-1$;
that \emph{$F$ is neutral} when  \eqref{alt1} holds both with $\si=1$ and $\si=-1$.
When $F$ is positive or neutral  $x^+(F)$  denotes the translation $x$ which appears in \eqref{alt1} when $\si=1$. When  $F$ is negative or neutral the vector $x$ which appears in \eqref{alt1} when $\si=-1$  is denoted by $x^-(F)$. 
The symbol $P'_F$ denotes $P'+x^+(F)$ when $F$ is positive and denotes  $-P'+x^-(F)$ when $F$ is negative.
\end{definition}
It is easy to check that the definition does not depend on the choice of $w$. Observe that the polytope $P'_F$ is a translation or reflection of $P'$ with the property that $F$ is a face of $P'_F$ and $\cone(P,F)=\cone(P'_F,F)$.

Within the rest of this section the terms boundary, interior and neighbourhood of a subset of $\pa P$ always refer to the relative topology induced on $\pa P$ by its immersion in $\Real^3$, with the exception of  $\pa P$ and $\pa P'$ which keep their original meaning. 
Moreover the term face will always mean proper face.
Let us try to express the global aim of the lemmas in this section in terms which are the least technical possible. Let $G$ be any face of $P$ which is not neutral.
We prove that there is a subset $\Sp$ of $\pa P\cap \pa P'_G$, which contains $G$,  to which it corresponds an ``antipodal'' subset $\Sm$ of  $\pa P\cap \pa P'_G$.  Moreover, if $G$ and $y\in\Real^3$ are suitably chosen then the boundaries of $\Sp$ and of $-\Sm+y$ coincide and  $\pa P$, $\pa P'_G$, $-\pa P+y$ and $-\pa P'_G+y$ all coincide in an one-sided neighbourhood of $\pa \Sp$. The set $\Sp$ is the  component of the intersection of $\pa P$ (with some exceptional faces removed) and $\pa P_G$ which contains $G$.

\begin{definition}\label{definizione_Sp} 
Let $\fz$ be a positive face  of $P$. Let $\cF$ be the collection of the edges or facets $F$ of $P$  which are  positive or neutral  and  satisfy $x^+(F)\neq x^+(\fz)$. Let
\begin{equation*}
\Sz=
\inte\Big( 
\pa P'_\fz\bigcap \pa P \setminus \bigcup_{F\in \cF}F\Big)
\end{equation*}
and let $\Sp$ be the closure of the  component of $\Sz$ which contains $\relint \fz$.
If $\fz$ is negative we define $\cF$, $\Sz$ and $\Sp$ as above, with positive substituted by negative and $x^+$ substituted by $x^-$.
\end{definition}
In the previous definition we have implicitly used the inclusion $\relint\fz\subset\Sz$, which is proved in next lemma.
\begin{lemma}\label{relintfz}
We have $\relint\fz\subset\Sz\cap\inte\Sp$.
\end{lemma}
\begin{proof}
It suffices to prove $\relint\fz\subset\Sz$, since this inclusion implies immediately $\relint\fz\subset\inte\Sp$. When $\fz$ is a facet the inclusion is obvious. Now assume that $\fz$ is a positive vertex $q$. Since $\cone(P,q)=\cone(P'_\fz,q)$,  $P$ and $P'_\fz$ coincide in a neighbourhood of $q$. Therefore to each face $F$  of $P$ containing $q$ it corresponds a face $F'$ of $P'_\fz$ containing $q$ with $\cone(P,F)=\cone(P'_\fz,F')$. If $F$ is positive or neutral then it necessarily coincides with $F'$ (a convex polytope has an unique face with a given support cone) and therefore $x^+(G)=x^+(q)$. This proves that no face of $P$ containing $q$ belongs to $\cF$, and this implies that  a neighbourhood of $q$ is contained  in $\Sz$. Similar arguments prove $\relint\fz\subset\Sz$  when $\fz$ is a negative vertex or  an edge.
\end{proof}

The next lemma proves that, when $P$ is not a translation or reflection of $P'$,  $P$ has both positive and negative facets (and thus $\Sp\neq\pa P$).

\begin{lemma}\label{centralmente_simmetrico}
If no facet of $P$ is negative (is positive) then $P'$ is a translate of $P$ (of $-P$, respectively). If each facet of $P$ is neutral then $P$ and $P'$ are centrally symmetric.
\end{lemma}
\begin{proof}
Assume that no facet of $P$ is negative.
Let $F_1$ and $F_2$ be adjacent facets of $P$ (in the sense that they have an edge $S$ in common). We prove that  $x^+(F_1)=x^+(F_2)$. Indeed, if $x^+(F_1)\neq x^+(F_2)$ then $F_2$ is not a face of $P'+x^+(F_1)$ and the facet of $P'+x^+(F_1)$ containing $S$ and different from $F_1$ has no corresponding facet in $P$.

Since each facet of $P$ can be joined to $F_1$ by a finite sequence of adjacent facets, we have $P=P'+x^+(F_1)$.
A similar argument proves the claim regarding the absence of positive facets. These two implications together show that when $P$ has only neutral facets  both $P$ and $P'$ are centrally symmetric.
\end{proof}

Standard arguments of general topology prove that $\Sp$, the closure of a connected open set, coincides with $\cl\inte\Sp$.  We recall  that for an open subset of $\Real^2$ (and thus also for an open subset of $\pa P$) being connected is equivalent to being path-connected. 
The set $\pa \Sp$ is clearly the union of finitely many segments, which we call edges.
\begin{lemma}\label{proprieta_frontiere} Assume that $P$ is not a translation or a reflection of $P'$, that  $\fz$ is positive and $x^+(\fz)=0$. Let $\se$ be an edge of $\pa\Sp$. The following assertions hold.
\begin{enumerate}
\item\label{pf_item1} There exist  a facet $F$ of $P$ and a facet $F'$ of $P'$, with $F\notin\cF$ and  $\inte F\cap F'\neq\emptyset$, such that $F\cap F'=F\cap \Sp$ and  $ \se$ is an edge of the polygon $F\cap F'$.
\item\label{pf_item2} There exists  $ \tdue= \tdue( \se)\in\Real^3$ such that the following properties hold:
\begin{enumerate}
\item\label{pf_item2a} If  $\se$ is an edge of $P$ and  of $P'$, then it  is an edge  of $-P+ \tdue$ and of $-P'+ \tdue$ too and
\begin{equation}\label{pf_coni_uguali}
\cone(P,\se)=\cone(- P'+ \tdue, \se)\neq\cone(P', \se)=\cone(-P+ \tdue, \se).
\end{equation}
In this case $S$ is negative and  $ \tdue=x^-(\se)$.
\item\label{pf_item2b} If  $\se$ is not an edge of $P$ or it is not an edge of $P'$, then $F$ and  $F'$ are also facets respectively of $-P'+ \tdue$ and $-P+ \tdue$.  In this case $F$ is negative and $ \tdue=x^-(F)$.
\end{enumerate}
\item\label{pf_item3} There exist a  neighbourhood $W$ of $\relint  \se$ such that
\begin{equation}\label{uguali_intorno}
\Sp\cap W\subset(-\pa P+ \tdue)\cap(-\pa P'+ \tdue).
\end{equation}
\end{enumerate}\end{lemma}

\begin{proof}
Claim~\eqref{pf_item1}. Let $z\in\relint \se$. There exists a sequence $(z_n)$ of points of $\pa P$ converging to $z$ and contained in the component of $\Sz$ containing $\relint\fz$. We may assume that infinitely many terms of this sequence are contained in  $\inte(F\cap F')$, where $F$ and $F'$ are suitable coplanar  facets of $P$ and $P'$, respectively, containing $\se$ and with $F\notin \cF$. Since $\inte(F\cap F')$ is contained in $\Sz$ and contains points path-connected to $\relint\fz$ by a path in $\Sz$, it is contained in the component of $\Sz$ containing $\relint\fz$. Thus $F\cap F'\subset\Sp$.
On the other hand, we have $F\cap\Sp\subset F\cap\pa P'$ (by definition of $\Sp$) and $F\cap\pa P'=F\cap F'$ (by the convexity of $P$ and $P'$). All these inclusions imply 
\begin{equation}\label{fintersectfp}
F\cap F'=\Sp\cap F.
\end{equation}
In view of \eqref{fintersectfp} and $\se\subset\pa\Sp$, $\se$ cannot intersect $\inte(F\cap F')$. Thus $\se$ is contained in an edge of the polygon $F\cap F'$. It is easy to prove that it coincides with such an edge.

Claim~\eqref{pf_item2}. We assume  that
\begin{equation}\label{l_comune}
\text{$\se$ is an edge of $P$ and $P'$}
\end{equation}
and  prove that
\begin{equation}\label{coni_diversi}
 \cone(P, \se)\neq\cone(P', \se).
\end{equation}
Assume \eqref{coni_diversi} false and let us prove that $\Sp$ ``extends on both sides of $\se$''.
Let $G$ be the facet of $P$ containing $\se$ and different from $F$, and let $G'$ be the  facet  of $P'$ containing $\se$ and different from $F'$. If \eqref{coni_diversi} is false, $G$ and $G'$ are coplanar and their intersection has non-empty interior and is path-connected to $\se$.
We prove that $\se\notin\cF$ and $G\notin\cF$. The assumption \eqref{l_comune} and the denial of \eqref{coni_diversi} imply  $\se$  positive or neutral and $x^+(S)=0$. Thus, $\se\notin\cF$.
Assume that $G$ is positive or neutral. In this case  both $G+x^+(G)$ and $G'$ are facets of $P'$ with the same outer normal and thus they coincide. Since $G$ and $G'$ share the edge $ \se$, it has to be $x^+(G)=0$. This proves $G\notin\cF$. 
Since $\se\notin\cF$ and $G\notin\cF$,  we have $(F\cap F')\cup\se\cup(G\cap G')\subset\Sp$. Thus  $\se\cap\inte \Sp\neq\emptyset$ and $\se$ is not contained in $\pa \Sp$. 

Assume again  \eqref{l_comune}. Condition \eqref{coni_diversi} implies that $ \se$ is negative. Choose $w\in S^2$ so that $S=P_w$ and define $ \tdue=x^-( \se)$. Conditions
\eqref{alt1} and \eqref{alt2} imply that
$ \se$ is a common edge of $-P+ \tdue$ and $-P'+ \tdue$ and \eqref{pf_coni_uguali} holds.

Now assume   \eqref{l_comune} false. In this case we have $F\neq F'$, which implies that  $F$ is negative. In fact,  if $F$ is positive or neutral then $x^+(F)\neq0$, because $F$ is not a facet of $P'$, but $x^+(F)\neq0$ contradicts $F\notin\cF$, since $x^+(\fz)=0$. Choose $w\in S^2$ so that $F=P_w$ and define $ \tdue=x^-(F)$. Conditions \eqref{alt1} and \eqref{alt2} imply that $F$ is a facet of $-P'+ \tdue$ and  $F'$ is a facet of  $-P+ \tdue$. 

Claim~\eqref{pf_item3} is an immediate consequence of Claims~\eqref{pf_item1} and~\eqref{pf_item2}.
\end{proof}

The number of components of $\pa \Sp$ is finite since the number of edges of $\pa\Sp$ is finite. Each component is a polygonal curve.

\begin{lemma}\label{om_costante}Assume that $P$ is not a translation or a reflection of $P'$, that  $\fz$ is positive and $x^+(\fz)=0$. Let  $C_1,\dots,C_s$, for a suitable $s>0$, denote the components of $\pa\Sp$ and let $m\in\{1,\dots,s\}$.
The vector $\tdue$ associated to each edge of $C_m$ by Lemma~\ref{proprieta_frontiere} does not depend on the edge and \eqref{uguali_intorno} holds true when $W$ is a suitable neighbourhood of $C_m$.
\end{lemma}
\begin{proof}
Let $ \se_1$ and $ \se_2$ be any two adjacent edges of $C_m$ and let $q$ be the common endpoint. For each $i=1,2$, let $F_i$,  $F_i'$ and $\tdue(\se_i)$  be the facets and vector associated to $\se_i$ by   Lemma~\ref{proprieta_frontiere}. 
To prove the lemma it suffices to prove that $\tdue(\se_1)=\tdue(\se_2)$ and that 
\begin{equation}\label{intorno_q}
\text{\eqref{uguali_intorno} holds when $W$ is a suitable neighbourhood of $q$,}
\end{equation}
since any two edges of $C_m$ are joined by a finite sequence of adjacent edges of $C_m$. We divide the proof in three cases. 
In the first two cases the claim follows from the existence of a suitable face containing $q$ whose vector $x^-$ coincides both with $y(\se_1)$ and $\tdue(\se_2)$. In the third case some topological arguments are needed. 

\textit{Case 1. The point $q$ is a vertex of $P$ and of $P'$.}

Assume $\cone(P,q)=\cone(P',q)$. In this case $q$ is positive or neutral and arguments similar to those used in the proof of Lemma~\ref{relintfz} show that a neighbourhood of $q$ is contained  in $\Sp$. This is impossible because $ \se_1$, $ \se_2\subset\pa\Sp$.

We may thus assume $\cone(P,q)\neq\cone(P',q)$. This implies that $q$  is negative and  that $P$ and $-P'+x^-(q)$ coincide in a neighbourhood of the common vertex $q$. If, for some $i$, $S_i$ is not an edge of  $P$ or it is not an edge of $P'$, then Lemma~\ref{proprieta_frontiere} \eqref{pf_item2b} applies and $F_i$ is a facet of $-P'+ \tdue( \se_i)$. 
This  implies  $ \tdue( \se_i)=x^-(q)$. Similar arguments prove the same identity when  $S_i$ is  an edge of  $P$ and of $P'$ and Lemma~\ref{proprieta_frontiere} \eqref{pf_item2a} applies.
In each case we have $ \tdue( \se_1)=x^-(q)= \tdue( \se_2)$. Claim~\eqref{intorno_q} follows from the inclusion $\Sp\subset\pa P\cap\pa P'$ and from the coincidence of $P$ and $-P'+x^-(q)$ and of $P'$ and $-P+x^-(q)$ in a neighbourhood of $q$.

\textit{Case 2. The point $q$ is in the interior of a facet of $P$ or of a facet of $P'$.}

Assume, for instance, that $q$ is in the interior of a facet of $P$. 
In this case $F_1=F_2$ and, since $F_1'$ and $F_2'$ are coplanar with $F_1$, we have $F_1'=F_2'$ too. Since neither $\se_1$ nor $\se_2$ are edges of $P$, Lemma~\ref{proprieta_frontiere} \eqref{pf_item2b} applies both to $\se_1$ and $\se_2$.
Lemma~\ref{proprieta_frontiere} implies the equality  $\tdue( \se_1)=\tdue( \se_2)=x^-(F)$ and it also implies \eqref{intorno_q}.

\textit{Case 3.  The point $q$ is in the relative interior of an edge (but of no facet) of $P$ or of $P'$.}

Assume, for instance, that $q$ is in the relative interior of an edge $R$ of $P$.

First we observe that $R$ is both an edge of $F_1$ and of $F_2$, because $q$ belongs both to $F_1$ and to $F_2$. In addition,  neither $\se_1$ nor $\se_2$ can be  edges of $P$. We may assume $F_1\neq F_2$, because when $F_1=F_2$ the proof can be concluded as in Case~2. We may also assume $R$ positive, because when $R$ is negative or neutral the lemma can be proved as in the second paragraph of  Case~1, with $R$ playing the role of $q$.

We claim that the point $q$  does not belong to the relative interior of an edge of $P'$. If it does and $R'$ is the edge, then $R'$ is an edge of $F_1'$ and $F_2'$ (it is proved as above) and it is collinear to $R$ (because $F_i$ is coplanar to $F_i'$, for each $i=1,2$). In this case $q$ belongs to the relative interior of the edge $R\cap R'$ of $F_1\cap F_1'$. This  contradicts the fact, proved in Lemma~\ref{proprieta_frontiere},  that $ \se_1$ is an edge of $F_1\cap F_1'$.

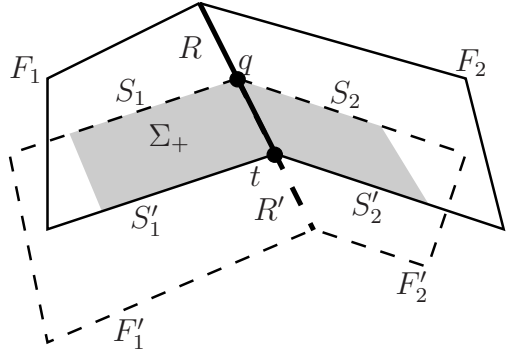
\begin{figure}
\begin{center}
\input{F_connessi.pstex_t}
\end{center}
\caption{The set $\Sp$ (in gray) in a neighbourhood of $[q,t]$. Solid lines represent edges of $\pa P$, while dotted lines represent edges of $\pa P'$.}
\label{fig_F_connessi}
\end{figure}
Let $R'=R+x^+(R)$. See Fig.~\ref{fig_F_connessi}. Observe that $R'$ is the edge of $P'$ in common to $F_1'$ and $F_2'$, because the facets of $P'$ which contain $R'$ have outer normals equal to those of $F_1$ and $F_2$, and $F_1'$ and $F_2'$ are the only facets of $P'$ with this property. Moreover $R$ and $R'$ are collinear and $x^+(R)\neq0$.
Thus, $R\in\cF$.

Let $R\cap R'=[q,t]$.  For $i=1,2$,  let $ \se'_i$ be the edge of $F_i\cap F'_i$ which contains $t$ and differs from $[q,t]$. It is evident that $\pa P$ differs from $\pa P'$ on one side of $ \se_i'$ and that these segments are contained in $\pa \Sp$. 
Let $C_k$ be the  component of $\pa \Sp$ which contains $ \se_1'$ and $ \se_2'$, and let us prove that $k\neq m$.
In a neighbourhood of $[q,t]$ the set $\Sp$ coincides with
\[
(F_1\cap F_1')\cup(F_2\cap F_2')
\]
and it is bounded on one side by $ \se_1\cup  \se_2$ and on the other side by $ \se_1'\cup  \se_2'$.  By definition of $\Sp$, for each $i=1,2$, there is a simple continuous curve $\Gamma_i$ contained in $\Sz$ and connecting a given point $p_i\in\inte(F_i\cap F'_i)$ to  a given point in $\relint\fz$. Since $R\in\cF$, we have $\Gamma_i\cap R=\emptyset$.  The union of $\Gamma_1$, $\Gamma_2$ and a continuous path connecting $p_1$ to $p_2$ and contained in $\inte(F_1\cap F'_1)\cup R\cup\inte(F_2\cap F'_2)$ disconnects $\pa \Sp$ in two sets, with $C_m$ in one set and $C_k$ in the other set. This implies $k\neq m$.

We say that two different components $C_h$ and $C_j$  of $\pa \Sp$ are  \emph{$\cF-$connected by  a segment $[a,b]$} if $a\in C_h$, $b\in C_j$, $[a,b]$ is contained in an edge of $P$ which belongs to $\cF$ and  $(a,b)\subset\inte\Sp$. What we have proved so far implies that if $ \tdue( \se_1)\neq \tdue( \se_2)$ then $C_m$ is  $\cF-$connected to another component of $\pa \Sp$ by a segment with endpoint $q$.

We  prove that any two different components $C_h$ and $C_j$ can be $\cF-$connected by at most one segment. Assume that they are $\cF-$connected by two different segments $[a,b]$ and  $[a',b']$, and consider a simple closed curve $\Gamma$ with
\[
\Gamma\subset C_h\cup C_j\cup [a,b]\cup[a',b'],\quad\text{and}\quad [a,b]\cup[a',b']\subset\Gamma.
\]
This curve  is not contained in $\Sz$ and it disconnects $\pa P$ in two  open non-empty sets such that one of them does not intersect $\Sp$. This contradicts the inclusion $(a,b)\cup (a',b')\subset\inte\Sp$. A similar argument proves that there is no finite sequence $i_1,\dots,i_p$, where $p>2$ and $i_l\neq i_j$ whenever $l\neq j$, such that $C_{i_p}$ is $\cF-$connected to $C_{i_{1}}$ and $C_{i_l}$ is $\cF-$connected to $C_{i_{l+1}}$, for each $l=1,\dots,p-1$. We call such a configuration a \emph{closed circuit of $\cF-$connected components}.
Observe that
\begin{equation}\label{diff_uguali}
 \tdue( \se_1)- \tdue( \se_2)=x^-(F_1)-x^-(F_2)= \tdue( \se_1')- \tdue( \se_2').
\end{equation}
The corresponding property holds for any pair of components $\cF$-connected  by $[d,e]$:   the difference between the vectors $\tdue$ of $C_h$ across $d$  equals the difference between the vectors $ \tdue$ of $C_k$ across $e$.

We conclude the proof of the lemma.
Put $i_1=m$ and $i_2=k$.
We prove that if $C_{i_2}$ is $\cF-$connected  only to $C_{i_1}$  then
\begin{equation}\label{om_uguali}
\tdue( \se_1)= \tdue( \se_2).
\end{equation}
In this case there is a sequence $R_1,\dots,R_p$ of  different consecutive edges of $C_{i_2}$ such that $ \se_1'=R_1$ and $\se_2'=R_p$. We have $\tdue(R_i)= \tdue(R_{i+1})$ for each $i$, because $C_{i_2}$ is $\cF-$connected  only to $C_{i_1}$ and the endpoint in common to $R_i$ and $R_{i+1}$ is different from $t$. 
This implies $\tdue( \se'_1)= \tdue( \se'_2)$ and  \eqref{om_uguali}.
Similar arguments prove \eqref{om_uguali} when there is a  component $C_{i_3}$ different from $C_{i_1}$ which is $\cF-$connected  only to $C_{i_2}$.
Therefore, when $ \tdue( \se_1)\neq \tdue( \se_2)$  it is possible to define an  infinite sequence $i_l$, $l=1,2,\dots$ such that  $C_{i_l}$ is $\cF-$connected to $C_{i_{l-1}}$. Each index in the sequence is different from all the previous ones, because if  $i_l=i_p$, for some $l$ and $p$ with $p\leq l$, then there is a closed circuit of $\cF-$connected components.  This construction contradicts the finiteness of the number of components of $\pa \Sp$. 

Claim \eqref{intorno_q} follows by Lemma~\ref{proprieta_frontiere}\eqref{pf_item2b} and the observations contained in the first paragraphs of Case~3 (see also Fig.~\ref{fig_F_connessi}).
\end{proof}

\begin{lemma}\label{jordaninverse}
Let $\Lambda\subset\pa P$ have connected interior and satisfy $\Lambda=\cl\inte\Lambda$, $\Lambda\neq\emptyset$ and $\Lambda\neq\pa P$. Assume that $\pa\Lambda$ is the union of finitely many polygonal curves (each with finitely many edges). The boundary of each component of $\pa P\setminus\Lambda$ is a closed simple curve. 
\end{lemma}

\begin{proof}
Let $A$ be a component of $\pa P\setminus\Lambda$. Clearly we have $\pa A\subset\pa \Lambda$. 

Let us prove that $\pa P\setminus \cl A$ is connected. It suffices to prove that each point $q\in \pa P\setminus \cl A$ is path-connected to  $\inte\Lambda$ by a continuous curve contained in  $\pa P\setminus \cl A$. If $q\in\pa\Lambda\setminus \cl A$ this is obvious due to the simple structure of $\pa \Lambda$. If $q$ belongs to some component $A'$ of $\pa P\setminus\Lambda$, with $A'\neq A$, there is a curve contained in $A'$ connecting $q$ to some point of $\pa \Lambda\setminus\cl A$. This curve can be extended to a curve contained in $\inte \Lambda\cup A'\cup\{q\}$ connecting $q$ to $\inte \Lambda$.

Consider $\pa\Lambda$ as a graph. If $[a_1,a_2]$ denotes an edge of $\pa\Lambda$  then there is at least another edge  of $\pa\Lambda$ different from $[a_1,a_2]$ and with endpoint $a_2$, because otherwise $\Lambda\neq\cl\inte\Lambda$. If in addition $[a_1,a_2]$  is contained in $\pa A$, then one of the edges of $\pa\Lambda$ with endpoint $a_2$ and different from $[a_1,a_2]$, say $[a_2,a_3]$, is necessarily an edge of $\pa A$. The iteration of this construction defines a closed simple curve $\gamma\subset\pa A$.

Let $\Gamma_1$ and $\Gamma_2$ be the components of $\pa P\setminus\gamma$. Since $A$ is connected and does not meet $\gamma$ we have either $A\subset\Gamma_1$ or $A\subset\Gamma_2$. Assume 
\[
A\subset\Gamma_1.
\]
Since $\pa P\setminus \cl A$ is connected, contains $\Gamma_2$ and does not meet $\gamma$, we have  $\pa P\setminus \cl A\subset\Gamma_2$, that is,
\[
\cl A\supset \cl\Gamma_1.
\]
The previous inclusions imply $\cl A= \cl\Gamma_1$. Since $A=\inte\cl A$ (it is an easy consequence of the identity $\Lambda=\cl\inte\Lambda$) and $\Gamma_1=\inte\cl \Gamma_1$ (it follows from the definition of $\Gamma_1$) the identity $\cl A= \cl\Gamma_1$ implies $A=\Gamma_1$. Therefore $\pa A=\gamma$.
\end{proof}
\begin{remark}
This lemma, which seems an inverse form of Jordan Curve Theorem, does not seem to be available in the literature.  Andreas Zastrow told us that it can be derived from Alexander Duality in Algebraic Topology; see~\cite[p.~179]{Gre}. It can be proved that this duality implies that the rank of the first homology group $H_1(\pa A)$ of $\pa A$, (with $\pa A$ thought as a graph embedded in $\pa P$) equals the number of the components of $\pa P\setminus \pa A$ minus $1$. Since, as proved above, $\pa P\setminus \pa A$ has two components, the rank of $H_1(\pa A)$ is $1$. This fact and the property $A=\inte\cl A$ imply that $\pa A$ is a simple closed curve.
\end{remark}

\begin{lemma}\label{globale_frontiera} 
Assume that $P$ is not a translation or a reflection of $P'$ and, for each $m=1,\dots,s$, let $\tdue_m$ denote the vector associated to $C_m$ by Lemma~\ref{om_costante}.
It is possible to choose $\fz$ so that $\tdue_i= \tdue_j$ for each $i,j=1,\dots,s$.
\end{lemma}

\begin{proof} Given a positive or negative face $G$ of $P$, let $\cF(G)$, $\Sz(G)$ and $\Sp(G)$ be as in Definition~\ref{definizione_Sp}, with $G$ substituting $\fz$. We  associate to $G$ a positive integer $s(G)$ as follows. Given  $B\subset\pa P$ define
\begin{gather*}
\size(B)=\text{number of facets of $P$ whose interior intersects $B$}\\
\intertext{and}
\siz(G)=\inf_{\text{$A$ component of $\pa P\setminus\Sp(G)$}}
\size(\pa P\setminus A).
\end{gather*}
Let $\fz$ denote a face which minimises $\siz(G)$ over all positive or negative facets of $P$.
Let $C_1,\dots,C_s$ be the components of $\pa\Sp(\fz)$ and let $y_i$ be the vector associated by Lemma~\ref{om_costante} to $C_i$. We claim that $y_i=y_j$ for each $i,j=1,\dots,s$. 

Assume the contrary and let $\Az$ be a component of $\pa P\setminus\Sp(\fz)$ which attains the infimum in the definition of $\siz(\fz)$. By Lemma~\ref{jordaninverse}, $\pa \Az$ is a simple closed curve contained in one of the components $C_m$. Assume for instance that $\pa \Az\subset C_1$ and let $i\in\{1,\dots,s\}$ satisfy 
\begin{equation}\label{om_diversi}
\tdue_1\neq \tdue_i.
\end{equation}
We prove that there exists a positive or negative face $G^\un$ of $P$ such that
\begin{equation}\label{size_Gamma}
\siz(G^\un)<\siz(\fz).
\end{equation}
After possibly substituting $P'$ with $P'_{\fz}$, we may assume $\fz$ positive and $x^+(\fz)=0$. 
Let us define $G^\un$. Choose an edge $ \se^\un$ of $C_i$ and let $F^\un$ be the facet of $P$ associated to $ \se^\un$ by Lemma~\ref{proprieta_frontiere}. When $\se^\un$ is an edge of $P$ and $P'$ we define $G^\un= \se^\un$, otherwise we define   $G^\un=F^\un$. By Lemma~\ref{proprieta_frontiere},  $G^\un$ is negative and  $x^-(G^\un)= \tdue_i$. Therefore $P_{G^\un}=-P'+y_i$. Moreover, by definition,
$\cF(G^\un)=\set{\text{edges or facets  $F$ of $P$ which are  neutral or negative and  satisfy $x^-(F)\neq \tdue_i$}}$. 
Let us prove 
\begin{equation}\label{Gamma_inclusi}
\Sp(G^\un)\cap \Az=\emptyset.
\end{equation}

Let $S$ be an edge of $\pa\Az$ and  let $F$  be the facet of $P$ associated to $S$ by Lemma~\ref{proprieta_frontiere}.
By Lemma~\ref{proprieta_frontiere}, either $S$ is an edge of $P$ and $P'$, is negative and $x^-(S)=y_1$, or else $F$ is negative and $x^-(F)=y_1$. Since $y_1\neq y_i$, in the first case we have $S\in\cF(G^\un)$, while in the second case we have $F\in\cF(G^\un)$. In both cases we have $S\cap\Si(G^\un)=\emptyset$. Thus we have $\pa \Az\cap\Si(G^\un)=\emptyset$. 
Since $\Sp(G^\un)$ is the closure of a component of $\Si(G^\un)$, the previous identity implies either   $\Sp(G^\un)\subset\cl\Az$ or \eqref{Gamma_inclusi}. Since $\Sp(G^\un)$ contains $\relint G^\un$ (by Lemma~\ref{relintfz}), and this set meets $C_i$, which does not intersect $\cl\Az$, we have \eqref{Gamma_inclusi}.

Let $A^\un$ be the component of $\pa P\setminus\Si(G^\un)$ which contains $\Az$. In order to prove  \eqref{size_Gamma} it suffices to prove $\size(\pa P\setminus A^\un)<\size(\pa P\setminus\Az)$. Assume the contrary, that is, in view of the inclusion $\Az\subset A^\un$, assume
\begin{equation*}
\size(\pa P\setminus A^\un)=\size(\pa P\setminus\Az).
\end{equation*}
Let $S$ and  $F$  be as above. The previous equality  implies  
\begin{equation}\label{size_uguali}
(\inte F)\cap \pa P\setminus A^\un\neq\emptyset,
\end{equation} since $(\inte F)\cap \pa P\setminus \Az\supset(\inte F)\cap \Sp(\fz)\neq\emptyset$. Let us prove the following claims:
\begin{enumerate}
\item\label{it_a} both $S$ and $-S+\tdue_1$ are edges of $P$, $P'$, $-P+\tdue_1$ and $-P'+\tdue_1$, condition \eqref{pf_coni_uguali} (with $y=\tdue_1$) holds  and we have $\pa \Az\cup(-\pa \Az+ \tdue_1)\subset\pa P\cap \pa P'$;
\item\label{it_b} the facet $F$ contains a translate of $\tdue_i- \tdue_1$;
\item\label{it_c} at each point of $\pa \Az\cup(-\pa \Az+ \tdue_1)$ a line parallel to $ \tdue_i- \tdue_1$ supports $P$ and $P'$;
\item\label{it_d}  there exists a neighbourhood of $\relint S$ such that in that neighbourhood we have
$\pa P=\pa P'$ on one side of $S$ and
$\pa P\cap \pa P'=\emptyset$ on the other side of $S$. 
The same property holds for some neighbourhood  of $\relint(-S+ \tdue_1)$. 
\end{enumerate} 

To prove \eqref{it_a}, let us show that $S$ is an edge of $P$ and of $P'$. If this is not true then, as proved above, we have $F\in\cF(G^\un)$ and, therefore, $F\cap\Si(G^\un)=\emptyset$. This implies $(\inte F)\cap\Sp(G^\un)=\emptyset$. Since $\inte F$ is clearly path-connected to $\Az$ (through $S$), we have $\inte F\subset A^\un$, which contradicts \eqref{size_uguali}.
The rest of \eqref{it_a} follows by Lemma~\ref{proprieta_frontiere}.
To prove \eqref{it_b} observe that, by \eqref{pf_coni_uguali}, $-P'+ \tdue_1$ has a facet $F''$ coplanar to $F$. If $p\in(\inte F)\cap \pa P\setminus A^\un$, then the segment connecting $p$ to a point $q$ of $S$ has to meet $\pa A^\un$, because  $q$ is an accumulation point of $\Az\subset A^\un$. A fortiori $F$ contains a segment $S'$ of $\pa \Sp(G^\un)$. By Lemma~\ref{proprieta_frontiere}, applied to $S'$ and to the pair $P$ and $P'_{G^\un}$, there is a facet of $P'_{G^\un}=-P'+\tdue_i$ coplanar to $F$. This facet necessarily coincides with $F''+\tdue_i-\tdue_1$.  This implies the claim. 
Claim \eqref{it_c} follows by \eqref{it_b} and the existence of a facet of $P'$ coplanar to $F$, which is proved in Lemma~\ref{proprieta_frontiere}. Claim \eqref{it_d} follows by \eqref{pf_coni_uguali}, with $y=\tdue_1$.

Let $\Delta$ be the cylindrical surface which supports $P$ and  with generatrix parallel to $ \tdue_i- \tdue_1$ and let $l$ denote a generic line contained in $\Delta$. We have $F\subset \Delta$ and $\pa\Az\cup(-\pa\Az+ \tdue_1)\subset \Delta$, by \eqref{it_b} and \eqref{it_c}. Moreover, Claim~\eqref{it_c}  implies
\begin{equation}\label{convessita_verticale}
[p,q]\subset \pa P\cap\pa P'\quad\text{whenever}\quad p,q\in l\cap \Big(\pa\Az\cup(-\pa\Az+ \tdue_1)\Big).
\end{equation}

Let  $f$ be an homeomorphism between $S^1$ and $\pa\Az$ and let $h=g\circ f$, where  $g:\Delta\to\Delta\mathbin| (y_1-y_i)^\perp$ is the orthogonal projection. The set $\Delta\mathbin| (y_1-y_i)^\perp$ is homeomorphic to $S^1$ and $h$ can be seen as a map from $S^1$ to $S^1$. Its  winding number is  $0$, $1$ or $-1$, since the curve $\pa\Az$ is simple. 

Assume that the winding number of $h$ is $0$. In this case $h$ is homotopic to a constant map. Since the projection $g$ is the identity between the first 
homotopy groups of $\Delta$ and of $\Delta\mathbin|(y_1-y_i)^\perp$, also $f$, as a map from $S^1$ to $\Delta$, is homotopic to the constant map. In this case $\Delta\setminus\pa\Az$ contains a bounded component $N$. Let us prove
\begin{equation}\label{spucln}
\Sp(\fz)=\cl N.
\end{equation}
The property \eqref{convessita_verticale} implies $N\subset\pa P\cap\pa P'$. If $l$ is as above and $l$ intersects transversally  the relative interior of an edge of $\pa\Az$, then $l\cap N$ is contained in the facet of $P$ and in the facet of $P'$ associated to this edge. Since their intersection is contained in $\Sp$, we have  $l\cap N\subset\Sp$. A continuity argument implies $\cl N\subset \Sp$. To prove the converse inclusion, observe that $\pa\Az\cap\Sz(\fz)=\emptyset$, because no  point in $\pa\Az$ has a neighbourhood contained in $\pa P\cap\pa P'$, by \eqref{it_d}. Therefore  $\Sp(\fz)$ is contained in $\cl N$ or in $\pa P\setminus N$. The inclusion $\cl N\subset \Sp(\fz)$ implies $\cl N= \Sp(\fz)$.

The previous argument also shows that $\Sp(\fz)$ is convex in the $y_1-y_i$ direction. This property and the connectedness of $\Sp(\fz)$ imply that  $\pa \Sp(\fz)$ has only one  component. This contradicts \eqref{om_diversi}.

Now assume  that the winding number of $h$ is $1$ or $-1$. Choose an orientation on $S^1$ and on $\Delta\mathbin| (y_1-y_i)^\perp$. We claim that $h$ is non-increasing or non-decreasing. If not, there is a line $l$ in $\Delta$ which intersects transversally  the relative interior of at least three segments $S_1$, $S_2$ and $S_3$ of $\pa\Az$. These segments are necessarily contained in the same facet of $\Delta$, and, if the point $S_2\cap l$ lies in between $S_1\cap l$ and $S_3\cap l$, then  $S_2\cap l$ is in the interior of the quadrilateral $\conv(S_1\cup S_2)$, which is contained in $\pa P\cap\pa P'$. This contradicts \eqref{it_d}. Let us prove that $h$ is strictly increasing or strictly decreasing. Assume the contrary. Then there exist consecutive segments $S_1$, $S_2$ and $S_3$ of $\pa\Az$, with $S_2$ parallel to the generatrix of $\Delta$ and $S_1$ and $S_3$ on opposite sides of the line containing $S_2$. Since both triangles $\conv(S_1\cup S_2)$ and $\conv(S_2\cup S_3)$ are contained $\pa P\cap\pa P'$ (by Claims~\eqref{it_a} and \eqref{it_c}), $P$ and $P'$ coincide on both sides of $S_2$. This contradicts Claim~\eqref{it_d}.

Thus each line $l\subset\Delta$ intersects $\pa \Az$ (and also $-\pa\Az+y_1$) in exactly one point. Arguing as in the proof of \eqref{spucln} one shows  that $\Sp(\fz)$ is the union of the segments parallel to $y_1-y_i$ with endpoints in $\pa \Az\cup(-\pa\Az+y_1)$. Thus $\pa \Sp(\fz)$ has at most two components $C_1=\pa \Az$ and $C_2=-\pa\Az+y_1$ and we have $i=2$. Each edge of $-\pa\Az+y_1=C_2$ is an edge of $P$ and $P'$, by Claim~\eqref{it_a}, and this implies that $G^1$ is an edge of $P$ and $P'$. Claim~\eqref{it_a} also implies that $G^1$ is an edge   of $-P'+y_1$ and 
\begin{equation*}\label{coni1}
\cone(P,G^1)=\cone(-P'+y_1,G^1).
\end{equation*}
Lemma~\ref{proprieta_frontiere} proves that $G^\un$ is also an edge of $-P'+y_2$ and that
\begin{equation*}\label{coni2}
\cone(P,G^1)=\cone(-P'+y_2,G^1),
\end{equation*}
since $G^\un$ is an edge of $C_2$.
The previous two identities imply $y_1=y_2$, contradicting \eqref{om_diversi}.
\end{proof}

\begin{lemma}\label{sigma_meno}
Assume that $P$ is not a translation or a reflection of $P'$. Let $\fz$ be a positive or negative face of $P$ such that the vectors associated by Lemma~\ref{om_costante} to each component of $\pa\Sp$ coincide, and let $y$ denote this vector.
Then there exists a  component of $\pa P \setminus(-\pa\Sp+y)$ intersecting $-\Sp+y$ whose boundary is $-\pa\Sp+y$. If $\Sm$ denotes the closure of this component then we have
\begin{equation}\label{formula_sigma_meno}
\Sm\subset\pa P\cap\pa P'_{\fz}\quad\text{and}\quad
\Sm\neq-\Sp+y.
\end{equation}
\end{lemma}

\begin{remark}\label{recall}We recall that $\Sp\subset\pa P\cap\pa P'_{\fz}$, by definition, and that $P$,  $P'$, $-(P_\fz)+y$ and $-(P'_\fz)+y$ coincide in an one-sided neighbourhood of $\pa\Sm=-\pa\Sp+y$, by Lemma~\ref{om_costante}.\end{remark}

\begin{proof}
Up to a translation or a reflection of $P'$,  we may assume $\fz$ positive and $\Pg=P'$. 
Choose $p_0\in\relint P\cap\relint(-P+y)$ and consider the homeomorphism from $-\pa P+y$ to $\pa P$ which associates to $p\in-\pa P+y$ the intersection of $\pa P$ with the ray issuing from $p_0$ and containing $p$. This homeomorphism maps $\inte(-\Sp+y)$ in a connected subset of $\pa P$ intersecting $-\Sp+y$ and whose boundary is $-\pa\Sp+y$.

Similar arguments prove that there exists a  component (whose closure we denote by $\Sm'$) of $\pa P' \setminus(-\pa\Sp+y)$ intersecting $-\Sp+y$ whose boundary is $-\pa\Sp+y$. 

Let us show that the set $V_P$ of the vertices of $P$ in $\Sm$ coincides with the set $V_{P'}$ of the vertices of $P'$ in $\Sm'$.  Let $w\in S^2$ be such that $P_w\in V_P$. Up to a perturbation of $w$, we may assume that $(-P+y)_w$ is a vertex too. Since $P$ and $-P+y$ coincide in an one-sided neighbourhood of $\pa \Sm=-\pa\Sp+y$ and they are convex, $(-P+y)_w$ is contained in $-\Sp+y$. Since $-\Sp\subset\pa(- P)\cap\pa(- P')$ we have $(-P)_w=(-P')_w$. The latter, the identity $DP=DP'$ and \eqref{setting} imply $P_w=P'_w$ and prove $V_P\subset V_{P'}$. Similar arguments prove $V_P\supset V_{P'}$. 

The identity $V_P=V_{P'}$ implies $\Sm=\Sm'$ and $\Sm\subset\pa P\cap\pa P'_{\fz}$.

Assume $\Sp\neq-\Sm+\tdue$ false.   Since $\fz\subset\Sp=-\Sm+\tdue$,  $\fz$ is also a face of $P'$ and of $-P'+y$. Moreover, the inclusion $\relint \fz\subset\inte \Sp$, proved in Lemma~\ref{relintfz}, implies $\cone(P,\fz)=\cone(P',\fz)=\cone(-P'+y,\fz)$. This proves that $\fz$ is neutral, contradicting the definition of $\fz$. 
\end{proof}

\section{Proof of the main theorem}\label{sec_mainproof}

\begin{proof}[Proof of Theorem~\ref{teorema0}]
We assume that $P'$ is not a translation or a reflection of $P$.  The pairs $(P,-P)$ and $(P',-P')$ are synisothetic, by Theorem~\ref{teorema_lmr}. Lemma~\ref{globale_frontiera} applies and proves that there exists a positive or negative face $\fz$ of $P$ such that the vectors associated by Lemma~\ref{om_costante} to each component of $\pa\Sp$ coincide. Let $y$ denote this vector and let $\Sm$ be defined as in Lemma~\ref{sigma_meno}. Up to a translation or a reflection of $P'$,  we may assume $\fz$ positive and $\Pg=P'$. We will prove $g_P\neq g_{P'}$.

In this proof the terms boundary and interior, when applied to $\Sp$ and $\Sm$, refer to the topology induced on $\pa P$ by its immersion in $\Real^3$. 

We claim   that either there exist  $q\in\inte (-\Sm+\tdue)$ vertex of  $-P+ \tdue$ and a plane $\pi$ strictly supporting $-P+ \tdue$ at $q$ such that  $\pi\cap P=\emptyset$, or else there exist  $q\in\inte \Sp$ vertex of $P$ and a plane $\pi$ strictly supporting $P$ at $q$ such that  $\pi\cap(-P+y)=\emptyset$. This follows by \eqref{formula_sigma_meno} and standard convexity arguments.
Indeed,  the set $V_P$ of the vertices of $P$ contained in $\inte \Sp$ differs from the set $V_{-P+y}$ of the vertices of $-P+y$ contained in  $\inte (-\Sm+y)$, because otherwise both $\Sp$ and $-\Sm+y$ are contained in the boundary of $\conv(\pa\Sp\cup V_P)$ and the inequality in \eqref{formula_sigma_meno} is false. If, say, $q\in V_{-P+y}\setminus V_P$ then let $\pi$ be a plane through $q$ strictly supporting $-P+y$. Up to a perturbation of $\pi$, we may assume that either $\pi$ does not intersect $P$ or it intersects  $\inte P$. In the first case we are done. In the second case there is an ``extremal'' vertex of $P$ contained in the open halfspace which is bounded by $\pi$ and does not contain $-P+y$. This vertex has the required properties.

Assume  $q\in\inte (-\Sm+\tdue)$ (in the other case the proof is similar)  and let $\pi$ be as above. Let $\pi_0$ be a plane which is parallel to $\pi$, intersects  $\pa\Sp$ in a point, say $z$ (up to a perturbation of $\pi$ we may always assume that $\pi_0\cap\Sp$ is a point), and such that $\pa\Sp$ and $q$ are on opposite sides of $\pi_0$. If $\pi_0^+$ denotes the closed halfplane bounded by $\pi_0$ and containing $q$, then we have $\pi_0^+\cap\pa P\subset\Sp$ and $\pi_0^+\cap(-\pa P+y)\subset\Sm$. Therefore, by the inclusion in \eqref{formula_sigma_meno}, we have
\begin{equation}\label{equality_piz}
\pi_0^+\cap P=\pi_0^+\cap P'\quad\text{and}\quad \pi_0^+\cap(- P+y)=\pi_0^+\cap(-P'+y).
\end{equation}

The plane $-\pi+y$ strictly supports $P$ and $P'$ at $-q+y$. Up to affine transformations, we may assume $-q+y=O$, $-\pi+y=\{x\in\Real^3 : x_3=0\}$, $P, P'\subset\{x : 0\leq x_3\leq1\}$, and we may also assume that the plane $\{x : x_3=1\}$ supports both $P$ and $P'$ (the existence of a common supporting plane follows by \eqref{equality_piz}). 
In this setting we have $y\cdot(0,0,1)>1$, because $y=q$ and the plane $\pi$, which contains $q$, does not intersect $P$ and $P'$. 
Let $e_3=(0,0,1)$ and let $N_1$ and $N_2$ denote the strips $\{x : z\cdot e_3\leq x_3\leq1\}$ and $\{x : 0\leq x_3\leq (-z+y)\cdot e_3\}$, respectively.
The identities \eqref{equality_piz} are equivalent to
\begin{equation}\label{proprieta_poliedri}
P\cap N_1=P'\cap N_1\quad\text{and}\quad
P\cap N_2=P'\cap N_2.
\end{equation}

We claim that $g_P$ and $g_{P'}$ differ in any neighbourhood of $z$.
Consider $P\cap (P+z+\ee)$ and $P'\cap (P'+z+\ee)$,
for $\ee\in\Real^3$, $\|\ee\|$ small and $\ee\cdot e_3<0$. 
These sets  are contained in the strip $N_1\cup N(\ee)$, where $N(\ee)=\{x : (z+\ee)\cdot e_3\leq x_3\leq z\cdot e_3\}$. Since $N_1\subset N_2+z+\ee$ (because $\ee\cdot e_3<0$ and $y\cdot e_3>1$), \eqref{proprieta_poliedri} implies $P=P'$ and $P+z+\ee=P'+z+\ee$ in $N_1$. In $N(\ee)$  we have $P+z+\ee=P'+z+\ee$, but $P$ and $P'$ differ. To be more precise, let  $A=\cone(P,O)=\cone(P',O)$, $C=\cone(P,z)\cap\set{x : x_3\leq 0}$ and $D=\cone(P',z)\cap\set{x : x_3\leq 0}$. We have 
\begin{equation*}
(P+z+\ee)\cap N(\ee)=(P'+z+\ee)\cap N(\ee)
=(A+z+\ee)\cap N(\ee).
\end{equation*}
We also have
\begin{equation*}
\begin{aligned}
P\cap (P+z+\ee)\cap N(\ee)&=(A+z+\ee)\cap (C+z),\\
P'\cap (P'+z+\ee)\cap N(\ee)&=(A+z+\ee)\cap (D+z),\\
\end{aligned}\end{equation*}
because $O$ is a vertex of $A$ and therefore the intersections in the left hand side of the previous formulas are contained in a neighbourhood of $z$.
Therefore 
\[
g_P(z+\ee)-g_{P'}(z+\ee)=g_{A,C}(\ee)-g_{A,D}(\ee).
\]
If we prove that $C\neq D$ and that $O$ is a vertex of $\conv(C\cup D)$, then Lemma~\ref{coni_finale_shortened} (with $B=\set{O}$) implies  $g_P\neq g_{P'}$.
Let us prove   
\begin{equation}\label{coni_diversi2}
\cone(P,z)\neq\cone(P',z).
\end{equation} 
By definition, $z$ is an endpoint of an edge $S$ of $\pa\Sp$ which intersects $\{x : z\cdot e_3=x_3\}$ only in $z$. 
Let $F$ and $F'$ be respectively the facets of $P$ and $P'$ associated to $S$ by Lemma~\ref{proprieta_frontiere}.
If \eqref{coni_diversi2} is false then either $z$ is a vertex of $P$ and $P'$, or $z$ is in the interior of $F$ and $F'$, or $z$ is in the relative interior of two edges of $P$ and $P'$. The last two possibilitie.s are ruled out by Lemma~\ref{proprieta_frontiere}, which proves that $z$ is a vertex of the polygon $F\cap F'$. The first possibility is ruled out by arguments similar to those used in the proof of Lemma~\ref{relintfz}. These arguments prove that when $z$ is a vertex the identity $\cone(P,z)=\cone(P',z)$ is not compatible with $z\in\pa \Sp$. 
Formula \eqref{coni_diversi2} and the equality  $\cone(P,z)\cap\set{x : x_3\geq 0}=\cone(P',z)\cap\set{x : x_3\geq 0}$ (a consequence of  \eqref{proprieta_poliedri}) imply $C\neq D$. 

It remains to prove that $O$ is a vertex of $\conv(C\cup D)$. 
The first condition in \eqref{proprieta_poliedri} implies that $C\cap \{x : x_3=0\}=D\cap\{x : x_3=0\}$. Therefore $O$ is not a vertex of $\conv(C\cup D)$ only if  $z$ is in the relative interior of a segment $R$ contained in $\pa P$, in $\pa P'$ and in $\{x : z\cdot e_3=x_3\}$. Let $S$, $F$ and $F'$ be as above.  Since $S\subset\pa P\cap\pa P'$ and $S\nsubseteq\{x : z\cdot e_3=x_3\}$,   $T:=\conv(R\cup S)$ is a triangle contained in $\pa P\cap \pa P'$.  The latter contradicts the property that $S$ is an edge of the polygon $F\cap F'$, because $\relint S\subset\relint T\subset F\cap F'$. 
\end{proof}


\section{About dimension $n\geq4$}\label{nmaggiore}
Each convex body has a representation as in \eqref{decomp_somma_diretta} in terms of directly indecomposable bodies $K_i$. Assume that at least two of the summands, say $K_1$ and $K_2$, are not centrally symmetric. In this case $(-K_1)\oplus K_2\oplus\dots\oplus K_s$ has the same covariogram of $K$ and is not a translation or reflection of $K$.  Does each convex body $L$ with $g_L=g_K$ have the same structure?

\begin{theorem}\label{prodotti_cartesiani}
Let $K\subset\Real^n$ be a convex body and let $K=K_1\oplus\dots\oplus K_s$ be its representation in terms of directly indecomposable bodies.
Assume that, for each $i=1,\dots,s$,  $E_i$ is a linear subspace containing $K_i$ and $E_1\oplus \dots\oplus E_s=\Real^n$. 
\begin{enumerate}
\item\label{cl_1} If $L$ is a convex body and $g_K=g_L$ then $L=L_1\oplus\dots\oplus L_s$, where, for each $i$, $L_i$ is a directly indecomposable convex body contained in $E_i$  and $g_{K_i}=g_{L_i}$.

\item\label{cl_2} If in addition $K$ is a polytope and, for each $i$, either $\dim K_i\leq3$, or $K_i$ is centrally symmetric, or $K_i$ is simplicial with $K_i$ and $-K_i$ in general relative position, then $L$ is a translation or reflection of $\si_1 K_1\oplus\dots\oplus\si_s K_s$, for suitable $\si_1,\dots,\si_s\in\set{+1,-1}$. 
\end{enumerate}
\end{theorem}

\begin{lemma}\label{irreducibility} A convex body $K$ is directly indecomposable if and only if $D\, K$ is directly indecomposable.
\end{lemma} 

\begin{proof}It is evident that if $K$ is not directly indecomposable then $DK$ has this property too. Vice versa, assume that $DK=L\oplus M$, with  $L\subset E_1$, $M\subset E_2$ convex sets of strictly positive dimension, and   $E_1$, $E_2$  linear subspaces with $E_1\oplus E_2=\Real^n$. Up to a linear transformation we may assume $E_1=E_2^\perp$. In this case we have $h_{DK}(x_1,x_2)=h_{DK}(x_1,0)+h_{DK}(0,x_2)$, for each $x_1\in E_1$ and $x_2\in E_2$. Moreover, if $K_1=K\mathbin| E_1$ and $K_2=K\mathbin| E_2$, we also have $h_{K_1}(x_1,0)=h_K(x_1,0)$ and $h_{K_2}(0,x_2)=h_K(0,x_2)$. The linearity of the support function with respect to Minkowski addition implies 
\begin{align*}
h_{K_1}(x_1,0)+h_{-K_1}(x_1,0)&=h_K(x_1,0)+h_{-K}(x_1,0)\\
&=h_{DK}(x_1,0)
\end{align*}
and  $h_{K_2}(0,x_2)+h_{-K_2}(0,x_2)=h_{DK}(0,x_2)$. 
These equalities imply
\begin{align*}
h_{K_1\oplus K_2}(x_1,x_2)+h_{-(K_1\oplus K_2)}(x_1,x_2)&=
h_{{DK}}(x_1,x_2)\\
&=h_{K}(x_1,x_2)+h_{-K}(x_1,x_2),
\end{align*}
which  can be rewritten as 
\begin{equation}\label{funz_supporto}
h_{K_1\oplus K_2}(x_1,x_2)+h_{K_1\oplus K_2}(-x_1,-x_2)
=h_{K}(x_1,x_2)+h_{K}(-x_1,-x_2).
\end{equation}
The inclusion $K\subset K_1\oplus K_2$, which is obvious, implies $h_K\leq h_{K_1\oplus K_2}$, and the latter inequality, together with \eqref{funz_supporto}, implies $h_K= h_{K_1\oplus K_2}$. Therefore $K=K_1\oplus K_2$ is not indecomposable too.
\end{proof}

\begin{proof}[Proof of Theorem~\ref{prodotti_cartesiani}] Claim~\eqref{cl_1}. The identity $D\,K=D\,L$, Lemma~\ref{irreducibility} and the uniqueness of the decomposition in directly indecomposable summands imply $L=L_1\oplus\dots\oplus L_s$, with $L_i\subset E_i$. The identity 
\[
\Pi_{i=1}^s g_{K_i}(x_i)=g_K(x_1,\dots,x_s)=g_{L}(x_1,\dots,x_s)=\Pi_{i=1}^s g_{L_i}(x_i),
\] 
valid for each $x=(x_1,\dots,x_s)\in E_1\oplus\dots\oplus E_s$, implies $g_{K_i}=\al_i g_{L_i}$, for each index $i$ and for suitable constants $\al_i$, $\al_i\neq0$. Lemma~\ref{lemma_sistema_cov} implies $\al_i=1$, for each $i$. 

Claim~\eqref{cl_2} follows from Claim~\eqref{cl_1}, Theorem~\ref{teorema0} and the positive results for the covariogram problem available in the literature and mentioned in the introduction.
\end{proof}

\begin{remark}\label{rem_facce_four} Let $P$ and $P'$  be convex polytopes in $\Real^4$ with non-empty interior and with equal covariogram, let $w\in S^3$ and assume that $P_w$ and $P_{-w}$ are facets.  Proposition~\ref{rufibach}, Lemma~\ref{lemma_sistema_cov} and Theorem~\ref{teorema0} imply that, possibly after a reflection or a translation of $P'$, we have $P'_w=\pm P_w$ and $P'_{-w}=\pm P_{-w}$. Contrary to the three-dimensional case (see also Remark~\ref{rem_facce}), the ambiguity due to the $\pm$ sign cannot be eliminated.  Indeed, if $K=\conv\{(0,-2),(0,2),(1,1),(1,-1)\}$, $L$ is a triangle,  $P=K\times L$, $P'=K\times(-L)$ and $w=(-1,0,0,0)$ then $P'_w= -P_w$ but there is no translation or reflection of $P'$ such that $P'_w= P_w$. 
\end{remark}

\bibliographystyle{amsplain}

\end{document}

%% file: parallelogrammi_5.pstex_t
\begin{picture}(0,0)%
\includegraphics{parallelogrammi_5.pstex}%
\end{picture}%
\setlength{\unitlength}{3522sp}%
\begingroup\makeatletter\ifx\SetFigFontNFSS\undefined%
\gdef\SetFigFontNFSS#1#2#3#4#5{%
  \reset@font\fontsize{#1}{#2pt}%
  \fontfamily{#3}\fontseries{#4}\fontshape{#5}%
  \selectfont}%
\fi\endgroup%
\begin{picture}(4544,2024)(-1191,-173)
\put(2746,299){\makebox(0,0)[lb]{\smash{{\SetFigFontNFSS{10}{12.0}{\familydefault}{\mddefault}{\updefault}$\mathcal K_2$}}}}
\put(1936,614){\makebox(0,0)[lb]{\smash{{\SetFigFontNFSS{10}{12.0}{\familydefault}{\mddefault}{\updefault}$\mathcal L_2$}}}}
\put(  1,839){\makebox(0,0)[lb]{\smash{{\SetFigFontNFSS{10}{12.0}{\familydefault}{\mddefault}{\updefault}$\mathcal K_1$}}}}
\put(  1,254){\makebox(0,0)[lb]{\smash{{\SetFigFontNFSS{10}{12.0}{\familydefault}{\mddefault}{\updefault}$\mathcal L_1$}}}}
\end{picture}%

%% file: parallelogrammi_tre.pstex_t
\begin{picture}(0,0)%
\includegraphics{parallelogrammi_tre.pstex}%
\end{picture}%
\setlength{\unitlength}{3646sp}%
\begingroup\makeatletter\ifx\SetFigFontNFSS\undefined%
\gdef\SetFigFontNFSS#1#2#3#4#5{%
  \reset@font\fontsize{#1}{#2pt}%
  \fontfamily{#3}\fontseries{#4}\fontshape{#5}%
  \selectfont}%
\fi\endgroup%
\begin{picture}(4229,1394)(-831,817)
\put(-79,1980){\makebox(0,0)[lb]{\smash{{\SetFigFontNFSS{11}{13.2}{\familydefault}{\mddefault}{\updefault}$\mathcal K_3$}}}}
\put(-710,1101){\makebox(0,0)[lb]{\smash{{\SetFigFontNFSS{11}{13.2}{\familydefault}{\mddefault}{\updefault}$\mathcal L_3$}}}}
\put(1641,1992){\makebox(0,0)[lb]{\smash{{\SetFigFontNFSS{11}{13.2}{\familydefault}{\mddefault}{\updefault}$\mathcal K_4$}}}}
\put(2083,1113){\makebox(0,0)[lb]{\smash{{\SetFigFontNFSS{11}{13.2}{\familydefault}{\mddefault}{\updefault}$\mathcal L_4$}}}}
\end{picture}%

%% file: F_connessi.pstex_t
\begin{picture}(0,0)%
\includegraphics{F_connessi.pstex}%
\end{picture}%
\setlength{\unitlength}{4144sp}%
\begingroup\makeatletter\ifx\SetFigFont\undefined%
\gdef\SetFigFont#1#2#3#4#5{%
  \reset@font\fontsize{#1}{#2pt}%
  \fontfamily{#3}\fontseries{#4}\fontshape{#5}%
  \selectfont}%
\fi\endgroup%
\begin{picture}(3134,2095)(654,-1223)
\put(675,393){\makebox(0,0)[lb]{\smash{{\SetFigFont{12}{14.4}{\familydefault}{\mddefault}{\updefault}$F_1$}}}}
\put(2698,-424){\makebox(0,0)[lb]{\smash{{\SetFigFont{12}{14.4}{\familydefault}{\mddefault}{\updefault}$S_2'$}}}}
\put(3312,435){\makebox(0,0)[lb]{\smash{{\SetFigFont{12}{14.4}{\familydefault}{\mddefault}{\updefault}$F_2$}}}}
\put(1300,257){\makebox(0,0)[lb]{\smash{{\SetFigFont{12}{14.4}{\familydefault}{\mddefault}{\updefault}$S_1$}}}}
\put(2571,243){\makebox(0,0)[lb]{\smash{{\SetFigFont{12}{14.4}{\familydefault}{\mddefault}{\updefault}$S_2$}}}}
\put(1288,-1178){\makebox(0,0)[lb]{\smash{{\SetFigFont{12}{14.4}{\familydefault}{\mddefault}{\updefault}$F_1'$}}}}
\put(2961,-898){\makebox(0,0)[lb]{\smash{{\SetFigFont{12}{14.4}{\familydefault}{\mddefault}{\updefault}$F_2'$}}}}
\put(1674,513){\makebox(0,0)[lb]{\smash{{\SetFigFont{12}{14.4}{\familydefault}{\mddefault}{\updefault}$R$}}}}
\put(2122,-452){\makebox(0,0)[lb]{\smash{{\SetFigFont{12}{14.4}{\familydefault}{\mddefault}{\updefault}$R'$}}}}
\put(1500,-10){\makebox(0,0)[lb]{\smash{{\SetFigFont{12}{14.4}{\familydefault}{\mddefault}{\updefault}$\Sp$}}}}
\put(2029,445){\makebox(0,0)[lb]{\smash{{\SetFigFont{12}{14.4}{\familydefault}{\mddefault}{\updefault}$q$}}}}
\put(2100,-254){\makebox(0,0)[lb]{\smash{{\SetFigFont{12}{14.4}{\familydefault}{\mddefault}{\updefault}$t$}}}}
\put(1392,-504){\makebox(0,0)[lb]{\smash{{\SetFigFont{12}{14.4}{\familydefault}{\mddefault}{\updefault}$S_1'$}}}}
\end{picture}%